\documentclass[12pt]{amsart}
\usepackage{graphicx}
\usepackage{enumitem}
\usepackage{color}
\usepackage{leftindex}
\usepackage{comment}

\usepackage[utf8]{inputenc}
\usepackage{latexsym, amsthm, verbatim, amsaddr}
\usepackage{hyperref, comment}
\usepackage{ amsmath, amssymb, colonequals}
\usepackage{tikz}
\usetikzlibrary{matrix, arrows, decorations.text}
\usepackage{euscript,mathrsfs, xcolor, enumitem}
\usepackage{lipsum}
\usepackage{geometry}
\newtheorem*{thm-num}{Theorem}

\newtheorem{theorem}{Theorem}[section]
\newtheorem{lemma}[theorem]{Lemma}
\newtheorem{proposition}[theorem]{Proposition}

\newtheorem{corollary}[theorem]{Corollary}
\numberwithin{equation}{section}
\theoremstyle{definition}
\newtheorem{definition}[theorem]{Definition}
\newtheorem{remark}[theorem]{Remark}

\def\Q{{\mathbb Q}}
\def\GL{{\mathrm{GL}}}			\def\R{{\mathbb R}}
\def\Z{{\mathbb Z}}

\def\f{{\mathbf f}}
				
				\def\N{{\mathbb N}}

            \def\0{{\mathbf 0}}

\def\bG{{\mathbf G}}

\newcommand{\sd}[1]{{\color{red} \sf SD: (#1)}}
\definecolor{cyan(process)}{rgb}{0.0, 0.72, 0.92}

\newcommand{\cov}{\operatorname{cov}}

\newcommand{\be}{\mathbf{e}}

\newcommand{\supp}{\operatorname{supp}}
\newcommand{\rk}{\operatorname{rank}}

\def\sl{\operatorname{SL}}
\def\eps{\varepsilon}
\newcommand{\norm}[1]{\left\|#1\right\|}
\def\tcr{\textcolor{red}}
\def\sda{\textcolor{blue}}
\def\int{\operatorname{int}}
\def\sing{\textbf{Sing}}

\theoremstyle{definition}

\numberwithin{equation}{section}

\begin{document}

\title{\sc On weighted singular vectors for multiple weights}
\begin{abstract}  
We introduce the notion of weighted singular vectors and weighted uniform exponent with respect to a set of weights. We prove invariance of these exponents for affine subspaces and submanifolds inside those affine subspaces. For certain analytic submanifolds, we show that there are totally irrational vectors with high weighted uniform exponent, extending the previously known existence results. Moreover, we show existence and non-existence of \textit{non-obvious} divergence orbits for certain cones.
\\
\smallskip

\noindent

\end{abstract}

\markboth {\hspace*{-9mm} 
\centerline{\footnotesize \sc Singular vectors with weights in a cone} }
{ \centerline{\footnotesize \sc Shreyasi Datta and Nattalie Tamam} \hspace*{-9mm}}

\author{\sc Shreyasi Datta and Nattalie Tamam}
\thanks{Department of Mathematics, University of York}
\thanks{Department of Mathematics, Imperial College London}
%\email{\it dattash@umich.edu}

\thispagestyle{empty}
% {\small \textbf{Keywords:} }
% \indent {\small {\bf 2000 Mathematics Subject Classification:} }

\maketitle
		
\section{Introduction}
A vector $x\in\R^d$ is called \emph{singular} if for every $\delta>0$ there exists $Q_0$ such that for all $Q\geq Q_0$ there are nonzero integer solutions $(p,q)\in \Z^d\times\N$ such that 
\begin{equation}\label{eq: non-weighted singular defn}
    \| qx-p\|<\frac{\delta}{Q},\quad q\leq Q,
\end{equation} 
where $
\Vert 
\cdot
\Vert $ denotes the sup norm in $\R^d.$
The set of singular vectors was first introduced by Khintchine in 1937 \cite{KHIN} in the setting of simultaneous approximation. 
It follows from Dirichlet theorem that any vector which lies on a rational hyperplane is singular, and so when searching for singular vectors, it is natural to exclude these cases. Vectors that do not lie on a rational hyperplane are called \emph{totally irrational}.  
Khintchine showed that when $d=1$ there are no totally irrational singular vectors, and when $d=2$ there exist totally irrational singulars. The later was later extended for any $d\ge2$ by Jarnik, \cite{Jar59}. 
Khintchine also showed that the set of singular vectors is of Lebesgue measure zero. These qualitative and quantitative results motivate the problems considered in this paper. 

One can consider a similar notion, replacing the norm in \eqref{eq: non-weighted singular defn} by a \emph{weighted quasi-norm}. 
The study of these weighted Diophantine approximations, initiated by Schmidt in \cite{Sch83}, is a central topic in metric number theory. Moreover, its connection with deep questions in homogeneous dynamics was explored and pointed out by Kleinbock in \cite{KleinbockDuke98}. See \cite{Dani2} for the connection between the two in the unweighted case, known as Dani's correspondence, as well as the interpretation of additional Diophantine properties; see also \cite{Weiss2004,KleinbockWeiss, DFSU}.
%For any vector $x\in \R^d$ we denote $x=(x_i)$, where $x_i\in\R$ is the $i$-th coordinate. 

Let us define the mentioned weighted quasi-norms. 
A vector $w=(w_i)\in[0,1]^d$ is called a \emph{weight} if it satisfies $
w_1+\dots+w_d=1$.  Each such weight defines a quasi-norm on $\R^d$ by 
\begin{equation}\label{eq: quasinorm defn}
    \norm{x}_w:=\max_{w_i\ne0}|x_i|^{1/w_i}\text{ for any }x=(x_i)\in\R^d,
\end{equation}
and we assign $\vert x\vert^{1/0}:=0$ for any $x\in (0,1)$. A weight $w$ is called a \emph{proper weight} it it belongs to $(0,1)^d$. We refer to $(1/d, \cdots, 1/d)$ as the standard weight.

 %Weighted Diophantine approximation date back to the work of Schmidt \cite{Sch69_badlinear} which considered weighted badly approximable linear forms. 
%In \cite{Kleinbock-1998} by Kleinbock presented the dynamical interpretation (also known as Dani's correspondence \cite{Da}) of badly and singular systems of linear forms. This work sprang a lot of interest in understanding many problems in the weighted set-ups \tcr{add citations}. 
%In this work we are interested in the weighted singular vectors and the weighted $\delta$-extremal vectors. 

\begin{definition}[$W$-singular vectors]\label{defn: singular}
Let $W\subset [0,1]^d$ be a set of weights. A vector $x\in\R^d$ is called \emph{$W$-singular} if for every $\delta>0$ there exists $Q_0>0$ such that for every $Q>Q_0$ and $w\in W$ there exists an integer solution $(p,q)\in\Z^d\times\N$ to the system of inequalities 
\begin{equation}\label{eq: singular defn}
    \norm{qx-p}_{w}\le\frac{\delta}{Q},\quad 0<q\le Q. 
\end{equation}
\end{definition}
We denote the set of $W$-singular vectors by $\sing(W)$. This is related to the action of a quasi-unipotent subgroup of $\mathrm{SL_{d+1}}(\R)$ with real eigenvalues on the space of unimodular lattices. For simplicity of notation, we denote $\sing(\{w\})$ by $\sing(w)$ and $\sing((1/d, \cdots, 1/d))$ by $\sing$ (note that $w=(1/d, \cdots, 1/d)$ gives the $d$-power of the sup-norm).  

It follows from the definition that for any set of weights $W$\begin{equation}\label{eqn: sing(W) and sing(w)}
   \Q^d\subset \sing(W)\subseteq \bigcap_{w\in W}\sing(w).
\end{equation}

\begin{remark}\label{remark:(0,1) singular}
A point $x=(x_i)\in\R^d$ is $(0,1)^d$-singular if and only if each $x_i$ is singular (in $\R$), see Lemma \ref{lem: closure and iff}. Since the only singular numbers are the rationals ones, $x$ is $(0,1)^d$-singular if and only if $x\in\Q^d$.
\end{remark}
%It is well known that for any $w$ $\sing(w)$ has Lebesgue measure zero. In \cite{KHIN}, Khintchine showed that there are uncountably many \emph{totally irrational} vectors in $\sing$, i.e., ones which do not lie on a rational hyperplane. In the recent years, the Hausdorff dimension of singular vectors and matrices, with and without weights, were calculated (see \tcr{add citations}), showing that `most' $w$-singular vectors are totally irrational.
%Remark \ref{remark:(0,1) singular} implies that this is no longer true when considering singulars with respect to sets of weights.  %Thus there are two kinds of problems that arise naturally. 
 %\begin{prob}
%Given any Borel probability measure $\mu$, determine if $\mu(\sing(W))=0$. 
 %\end{prob}
% Moreover,
 %\begin{prob}
% Given any Borel probability measure $\mu$, does $\bigcap_{w\in W}\sing(w)\cap\supp\mu$ contains totrally irrational vectors?
 %\end{prob}

Many classical results hold in the weighted setting. For example, the following is a weighted version of Dirichlet's theorem which follows from Minkowski's theorem; see \cite[Theorem 1.1]{KRao22}.

\begin{thm-num}[Weighted Dirichlet Theorem]\label{thm: Dirichlet}
    For any weight $w$, $x\in\R^d$, and a positive integer $Q$, there exist $q\in\Z$, $p\in\Z^d$ s.t. 
    \begin{equation}\label{eq: dirichlet}
        \norm{qx-p}_w<\frac{1}{Q},\quad 1\le q\le Q.
    \end{equation}
\end{thm-num}
It was also shown in \cite{KW-JMD-08}, that for almost every $x\in\R^d$ the constant $1$ on the right side of \eqref{eq: dirichlet} can not be improved to $c<1.$ Moreover, for any $\psi$ approximating function, set of weighted \emph{Dirichlet $\psi$ improvable} vectors  was studied in \cite{KSY22}. For the recent developments regarding weighted singular vectors, see \cite{DFSU, LSST, KimPark, KMW, KW} and the references therein. 

The main goal of this paper is to study both quantitative and qualitative results regarding $W$-singular vectors and also weighted uniform exponents.

%In \cite{KW, KMW} authors study the set of weighted singular vectors inside submanifolds in $\R^n$ and certain fractals in $\R^n$. 

%Singular vectors and their weighted versions are still less understood, and have been a topic of interest in many recent works; for example, see \cite{DFSU, LSST, KimPark, KMW, KW} and the references therein.
 
%\tcr{doesn't belong here:}In \cite{cheung_2007}, and \cite{Chevallier}  the Hausdorff dimension of singular vectors in $\R^d$ is shown to be $\frac{d^2}{d+1}$, which extends the observations stated above.

%One of the very compelling problems in the last few decades was to show that the intersection of different sets of \textit{weighted badly approximable vectors} is nonempty; a conjecture by Schmidt in \cite{Sc}. After withstanding attacks for over a few decades, this conjecture was settled in \cite{BPV2011}. Like badly approximable vectors, another number theoretic set that interests people is the set of \textit{singular vectors}, which was first introduced by Khintchine in 1920, \cite{KHIN,Ca}. 
%Singular vectors are related to the divergent orbits under a certain action by unipotent in the space of unimodular lattices. 

%Motivated by the settlement of Schmidt conjecture for weighted badly approximable vectors and many new developments in the case of the intersection of various weighted sets in \cite{Beresnevich2015, BNY21},
%in this paper, we concentrate on the intersection of \textit{weighted singular vectors}. More generally, we define and study \textit{$W$-singular vectors
%}, where $W$ is a set of weights. 

\subsection{Qualitative results: Existence of singular points}
Our first result extends the main theorem in \cite{KMW}, showing that any submanifold of dimension at least 2 in $\R^d$ conatins uncountably many totally irrational $w$-singular vectors for any proper weight $w$. In \cite{KW}, $w$-singular vectors were defined for any proper weight $w$. The next result shows even intersection of all proper weights contains uncountable many totally irrational vector.
Note that the only $(0,1)^d$-singular vectors are the rational ones (see Remark \ref{remark:(0,1) singular}), implying that a stronger result does not hold.

\begin{theorem}\label{Main: existence manifold}
Suppose $\mathcal{M}$ is a connected real analytic submanifold in $\R^n$ that is not contained inside any rational affine hyperplane, and $\dim(\mathcal{M})\geq 2.$ Then there are uncountably many vectors $x\in\mathcal{M}$ that do not lie on a rational hyperplane of $\R^n$ and are $w$-singular for any $w\in(0,1)^n$. 
\end{theorem}

While uploading this work, we saw a current preprint \cite{KMWW} by Kleinbock, Moshchevitin, Warren and Weiss who considered singular vectors with multiple weights in the case of matrices and submanifolds of matrices. In \cite[Theorem 4.6]{KMWW} authors consider the above theorem in more general set-up of matrices and general approximating function.

Next we recall the definition of uniform exponents with weights.

\begin{definition}\label{defn: ordi_uni}
For any weight $w$ and $x\in\R^d$ we denote by $\hat{\sigma}_w(x)$ the \emph{uniform $w$-exponent of $x$} to be the supremum of the real numbers $\sigma$ such that for all large $Q$, the system of inequalities \[
\norm{qx-p}_w<\frac{1}{Q^{\sigma}},\quad  0<q\leq Q\]
has a solution $(p,q)\in\Z^d\times\Z$. Denote by $\hat{\mathcal{W}}_{w,\sigma}$ the subsets of $\R^d$ with $w$-uniform exponent greater or equal to $\sigma$, respectively. %Denote by $\hat{\mathcal{W}}_{w,\sigma}^*$ the subset of $\hat{\mathcal{W}}_{w,\sigma}$ which do not intersect with any rational hyperplanes of $\R^d$, respectively. 
As before, we omit $w$ from the notation when considering the weighted norm which is propositional to sup-norm. 
\end{definition}

There are some simple observations about the uniform $w$-exponent. For a weight $w=(w_1,\dots,w_d)$ and a totally irrational $x\in \R^d$, it is easy to see that \[
1\leq \hat\sigma_w(x)\leq \left(\max_{1\le i\le d} w_i\right)^{-1}. \]
Additionally, and $\hat\sigma_w(x)>1$ implies $x\in\sing(w)$. In \cite{KMW} it is shown more generally that when an analytic submanifold $\mathcal{M}$ has dimension greater than $2$ and is not contained inside any rational hyperplane, then there are uncountably many totally irrational $x\in \mathcal{M}$ such that $\hat\sigma_{w}(x)\geq \left(1-\min_{1\le i\le } w_i\right)^{-1}$; see \cite[Corollary 1.5]{KMW}. 
In recent years, there has been a lot of interest in exploring uniform exponents, and, more generally, uniform approximations; see \cite{CGGM2020, YCC, Kleinbock_Nick_Compo, Kimkim22, KSY22, KRao22IMRN, KRao22}, and the references therein.

\begin{theorem}\label{thm: existance of very singular}
    Suppose $w$ is a proper weight such that $w_1\le w_2\le \cdots\le w_d$ and $\mathcal{M}$ is an analytic sub-manifold of $\R^d$ of dimension $2\le k\le d$ which is not contained in any rational hyperplane. That is, $\mathcal{M}$ is not contained in any rational affine hyperplane of $\R^d$. Then there are uncountably many totally irrational vectors in $\mathcal{M}\cap \hat{\mathcal{W}}_{w, \left(\sum_{i=k}^{d}w_i\right)^{-1}}$.
\end{theorem}

\begin{remark}
For Theorem \ref{thm: existance of very singular} we have the following: 
\begin{itemize}
    \item For the standard weight, it implies that any analytic sub-manifold of dimension $2\le k\le d$ in $\R^d$ has an uncountable intersection with $\hat{\mathcal{W}}_{w,\frac{d}{d-k+1}}$.  Note this standard weighted case for a general approximating function is also disscussed in \cite[Thm 1.13] {KMWW}.
    \item It improves the exponent in \cite[Corollary 1.5]{KMW} when the dimension of the manifold $\mathcal{M}$ is of dimension $k>2$.
    \item The order assumption on the coordinates of $w$ in Theorem \ref{thm: existance of very singular} is to insure that $\sum_{i=k}^{d}w_i$ is the largest summation of $d-k$ entries of the coordinates of $w$, and can be replaced by such assumption. If one wants to consider all manifolds $\mathcal{M}$, then this condition is necessary, as can be seen by Proposition \ref{prop: no intersection}.  However, it can be removed by assuming that $\mathcal{M}$ is not contained in an affine subspace that is parallel to certain axes, see Remark \ref{rem: no assumption on coord}. The  same remark also implies that Theorem \ref{Main: existence manifold} fails when considering points in $\hat{W}_{w,\delta}$ for a large enough $\delta.$
\end{itemize}
\end{remark}

\subsection{Quantitative results: Inheritance}
In 1960, Davenport and Schmidt showed that almost every $x\in \R$, $(x,x^2)$ is not in $\sing,$ which was later extended for any \textit{nondegenerate} submanifolds in $\R^n$, and more generally for $\sing(w)$ in \cite{KW}; see \S\ref{nondeg} for the definition of nondegeneracy. By \eqref{eqn: sing(W) and sing(w)} it follows that $\sing(W)$ also has measure zero inside any nondegenerate submanifold in $\R^n.$ In \cite{DX}, it was shown that the measure zero property of $\sing$ is inherited by a nondegenerate manifolds in $\R^n$ from its ambient affine space. Given any set of weights $W$ we defined $w_{min}$ as in \eqref{eq: condition on W}. In the results of this section, we assume that $w_{min}>0$.

Our first main theorem address an inheritance result for $W$-singular vectors.
\begin{theorem}\label{thm inheritance_singular}
%Let $W$ be a closed subset of $(0,1)^d$. 
Suppose $\mathcal{M}$ is a nondegenerate submanifold of an affine subspace $\mathcal{L}\subset\R^n$. Then 
    the following  are equivalent: 
    \begin{itemize}
    \item There exists $y\in \mathcal{L}$ which is not $W$-singular.
     \item There exists $y\in\mathcal{M}$ which is not $W$-singular.
        \item $\lambda_{\mathcal{L}}$--almost every $y\in\mathcal{L}$ is not $W$-singular.
        \item $\lambda_{\mathcal{M}}$--almost every $y\in\mathcal{M}$ is not $W$-singular.
    \end{itemize}
    Here $\lambda_{\mathcal{L}},\lambda_{\mathcal{M}}$ are the Lebesgue measures on $\mathcal{L}$ and $\mathcal{M},$ respectively.
\end{theorem}

Readers are referred to \S \ref{nondeg} for the  definitions of nondegeneracy in the above theorem. 
More generally than $W$-singular vectors we define the finer notion of weighted uniform exponent.
\begin{definition}[$W$-uniform exponents]\label{chihatset}
    For a vector $x\in\R^d$ we denote $\hat\sigma_{W}(x)$ to be the supremum of real numbers $\varepsilon$ such that there exists $Q_0>0$, and for every $Q>Q_0$ and $w\in W$ there exists an integer solution $(p,q)\in\N^d\times\Z$ to the system of inequalities 
\begin{equation}\label{eq: uniform defn}
    \norm{qx-p}_{w}\le\frac{1}{Q^\varepsilon},\quad 0<q\le Q. 
\end{equation} 
\end{definition}

It follows from the definition that for any $x\in\R^d$, $\hat\sigma_W(x)>1$ implies that $x$ is $W$-singular.  
For any Borel measure $\mu$ on $\R^d$, let us define $$\hat\sigma_W(\mu):=\sup\{\varepsilon~|~\mu(\{x~|~\hat\sigma_W(x)>\varepsilon\})>0\}.$$
For any submanifold $\mathcal{M}$ in $\R^n$ we define $\hat\sigma_W(\mathcal{M}):= \hat\sigma_W(\lambda_\mathcal{M}),$  where $\lambda_{\mathcal{M}}$ is the Lebesgue measures on $\mathcal{M}$.
Let us recall $\delta$ from \eqref{defn:delta}, $a_{w,t}$ and $\pi(u_x)$ as in \S \ref{sec: dani's correspondence}. We define, 
\begin{definition}\label{tauhatset-dynamics}
Given a set of weights $W$ and a vector $x\in\R^d$, we define $$
\hat\tau_W(x):= \liminf\limits_{t\to\infty} \inf\limits_{w\in W}\frac{-1}{t}\log\delta(a_{w,t} \pi(u_x)).$$ 
\end{definition}
Essentially $\hat\tau_W(x)$ is the rate at which $\mathcal{A}_W^+\pi(u_x)$ diverges. Similarly as before, we define $\hat\tau_W(\mu)$ and $\hat\tau_W(\lambda_\mathcal{M})$. See Theorem \ref{thm: Dani's correspondence:unirofm} for the relation between $\hat\tau_{W}(x)$ and $\hat\sigma_{W}(x)$.
%For the simplicity of the notation, we denote $\hat\sigma_w(\cdot)$ instead of $\hat\sigma_{\{w\}}(\cdot)$, and in the case $w$ is the standard weight, we simply denote $\hat\sigma(\cdot)$. 

More generally than Theorem \ref{thm inheritance_singular}, we have the following inheritance of  exponents $\hat\tau_W(\cdot)$.
\begin{theorem}\label{thm inheritance_uniform}
%Let $W$ be a closed subset of $(0,1)^d$. 
Suppose $\mathcal{M}$ is a nondegenerate submanifold of an affine subspace $\mathcal{L}\subset\R^n$. Then 
    \begin{equation}
        \hat\tau_W(\mathcal{M})=\hat\tau_W(\mathcal{L})= \inf\{\hat\tau_W(x)~|~x\in \mathcal{M}\}= \inf\{\hat\tau_W(x)~|~x\in \mathcal{L}\}.
    \end{equation}
\end{theorem}
Combining the above with Theorem \ref{thm: Dani's correspondence:unirofm}, and specializing in the standard weight case we get the following:
\begin{corollary}
   Suppose $\mathcal{M}$ is a nondegenerate submanifold of an affine subspace $\mathcal{L}\subset\R^n$. Then 
    \begin{equation}
        \hat\sigma(\mathcal{M})=\hat\sigma(\mathcal{L})= \inf\{\hat\sigma(x)~|~x\in \mathcal{M}\}= \inf\{\hat\sigma(x)~|~x\in \mathcal{L}\}.
    \end{equation}
\end{corollary}
In the weighted case, we can also get the following equivalence result using Remark \ref{use_critical} and Theorem \ref{thm inheritance_uniform}.
\begin{corollary}
Suppose $\mathcal{M}$ is a nondegenerate submanifold of an affine subspace $\mathcal{L}\subset\R^n$. Then 
    the following  are equivalent: 
    \begin{itemize}
    \item $\hat\sigma_W(\mathcal{M})=1.$ 
     \item $\hat\sigma_W(\mathcal{L})=1.$ 
        \item $\inf\{\hat\sigma_W(x)~|~x\in \mathcal{M}\}=1$.
        \item $\inf\{\hat\sigma_W(x)~|~x\in \mathcal{L}\}=1$.
    \end{itemize}    
\end{corollary}

%Note that $\hat\sigma_W\geq 1$ for any set of weights $W$, and $\hat\sigma_W=1$ if and only if $w_{min}=0.$ 
  
\begin{remark}
    Even though Theorem \ref{thm inheritance_singular} and Theorem \ref{thm inheritance_uniform} are closely related and, as we will see, their proofs are quite similar, none of them imply the other.
\end{remark}
\begin{remark}
Note that due to not having an exact formula relating $\hat\sigma_W(x)$ and $\hat\tau_{W}(x)$ for general $W$, Theorem \ref{thm inheritance_uniform} can not be directly transferred to an equivalent statement about $\hat\sigma_W(x).$ 

\end{remark}

\subsection{Existence and non-existence of diverging orbits}\label{sec:diverging orbits}

As earlier mentioned, the singular vectors correspond to divergent unipotent orbits of a one-parameter diagonal flow in the space of unimodular lattices, which was first discovered by Dani \cite{Da}. In \cite{Weiss2004}, Weiss started a study of understanding \textit{obvious} and \textit{non-obvious} orbits for actions of multi-dimensional groups and semigroups; see Definitions \ref{def: obvious} and \ref{def: divergence}. In the case of $\sing,$ the divergent orbits corresponding to totally irrational singular vectors are non-obvious. 

Let $\bG$ be a $\Q$-algebraic Lie group, $G=\bG(\R)$, $\Gamma=\bG(\Z)$, and $X=G/\Gamma$. 
Let $\pi$ be the natural projection $G\rightarrow X$. We refer to \S \ref{div_cone} for relevant definitions of \textit{divergence} and \textit{obvious} divergence. In \cite{Weiss2004} it was proved that the obvious divergence implies the standard divergence.

Let $S$ be the maximal $\Q$-split torus of $G$, $r:=\operatorname{rank}_{\Q}G$ and $\chi_1,\dots,\chi_r$ be the $\Q$-fundamental weights of $G$ (see the definition of weights representation in \S\ref{sec: fundamental representations}). Let 
\begin{equation}\label{eq: defn of A+}
    S^+:=\left\{s\in S: \forall i,\quad \chi_i(s)>0\right\}. 
\end{equation}
%and for any $1\le i\le r$, $\eps>0$ let \[S_{i,\eps}:=\left\{s\in S:\chi_i(s)\ge\eps\right\}.\]

%\begin{theorem}[{\cite[Thm 1.7]{tamam_2021}}]\label{thm: exists non-obvious}
%    If $A$ is a cone which does not contain $S^+$ for any choice of simple system, then there exists a non obvious diverged orbit for $A$ in $X$. 
%\end{theorem}

Our methods allow as to deduce the following more general version of Theorem \ref{Main: existence manifold}, which can also be viewed as stronger version of \cite{Weiss2004,Tamam}.

\begin{theorem}\label{thm: exists one-parameter diverges}
    Assume $\operatorname{rank}_{\Q}G\ge2$. Then, there exist uncountably many points $x\in X$ so that for any one-parameter subsemigroup $S':=\{s_t\}\subset S^+$, the orbit $S'x$ diverges in a non-obvious way. 
\end{theorem}

Note that it follows from \cite{ST} that the assumption about the $\Q$-rank of $G$ in the above result is necessary.  
The `tightness' of Theorem \ref{Main: existence manifold} which follows from the fact that the only $(0,1)^d$-singular vectors are the rational ones, also holds in the general case by the following claim. 

\begin{theorem}\label{thm: only obvious}
    %Let $S^+\subset A\subset T$. Then, 
    Any divergent orbit of $S^+$ diverges in an obvious way. 
\end{theorem}

\begin{remark}
    Note that by Theorem \ref{thm: only obvious} for any $x$ which satisfies the conclusion of Theorem \ref{thm: exists one-parameter diverges}, the orbit $A^+x$ does not diverge. 
\end{remark}

\section{Notation and preliminary results}\label{sec: notation}

In this section, we want to list some of the definitions and notations that we use in this paper. 

For any set of weights $W$, we define \begin{equation}\label{eq: condition on W}    
w_{\min}:=\inf \{w_i: i=1\cdots,d, w=(w_i)\in W\}.
\end{equation} Note that $w_{min}>0$ if and only if the closure $\overline{W}$ is a subset of proper weights. Also, let,
\begin{equation} \label{eq:wmax}   
w_{\max}:=\sup \{w_i: i=1\cdots,d, w=(w_i)\in W\}.
\end{equation} If $w_{min}>0$ then $w_{\max}<1.$
%Also let us define the following sets.
% \begin{definition}
%Given $\varepsilon>0$, a vector $x\in\R^d$ is called  \emph{${(\varepsilon,W)}$- singular} if there exists $Q_0>0$ such that for every $Q>Q_0$ and $w\in W$ there exists an integer solution $(p,q)\in\N^d\times\Z$ to the system of inequalities 
%\begin{equation}\label{eq: uniform defn}
   % \norm{qx-p}_{w}\le\frac{1}{Q^\varepsilon},\quad 0<q\le Q. 
%\end{equation} 
%We denote the set of ${(\varepsilon,W)}$- singular vectors by \emph{${\varepsilon(W)}$}. 
%\end{definition}
%Note that $
%\hat\sigma_W(x)=\sup\{\varepsilon~|~ x\in \varepsilon(W)\},
%$ where $\hat\sigma_W(\cdot)$ is defined as in \ref{chihatset}.
\subsection{Real analytic manifolds}
\label{sec: analytic manifolds}
Let $k\le d$, and let $U\subseteq \R^k$ be open. We say that $g: U\rightarrow \R^d$ is \emph{real analytic immersion} if it is injective, each of its coordinate functions $g_i:U\rightarrow \R$, $i = 1,\dots,d$ is infinitely differentiable, the Taylor series of each $f_i$ converges in a neighborhood of every $x\in U$, and the derivative mapping $d_xg: \R^k \rightarrow \R^d$ has rank $k$. 
By a \emph{$k$-dimensional real analytic sub-manifold} in $\R^d$ we mean  a subset $M \subseteq \R^d$ such that for every $\xi\in M$ there is a neighborhood $V\subseteq \R^d$ containing $\xi$, an open set $U\subseteq\R^k$, and a real analytic immersion $g : U \rightarrow \R^d$ such that $V\cap M= g(U)$.

The following is a useful property of real analytic sub-manifolds. 
\begin{lemma}\label{lem:k-dim manifold}
Let $\mathcal{M}_1$, $\mathcal{M}_2$ be real analytic sub-manifolds of $\R^d$ equipped with the inherited topology. If the intersection $\mathcal{M}_1\cap\mathcal{M}_2$ has nonempty interior in $\mathcal{M}_1$, then this intersection is open in $\mathcal{M}_1$; and thus, if additionally $\mathcal{M}_1$ is connected and $\mathcal{M}_2$ is closed, then $\mathcal{M}_1\subseteq \mathcal{M}_2$.
\end{lemma}

The following is a higher dimensional version of \cite[Prop. 3.2]{KMW}. 

\begin{lemma}\label{lem: k-dim manifold}
Let $\mathcal{M}\subseteq\R^d$ be a bounded real analytic manifold of dimension $k$, and let $A$ be an aﬃne-hyperspace such that $\mathcal{M}\not\subseteq A$. Then, $\mathcal{M}\cap A$ is a finite union of real analytic connected submanifold of $\R^d$ with dimension of at most $k-1$.
\end{lemma}

\begin{proof}
Let $\mathcal{M}_0:=\mathcal{M}\cap A$. Clearly $\dim \mathcal{M}_0\le k$. First, note that $\dim\mathcal{M}_0=k$ implies that $\mathcal{M}_0$ is open in $\mathcal{M}$, but also closed since $A$ is a closed subset of $\R^d$. By connectedness this would imply $\mathcal{M}\subseteq A$, contrary to assumption.
Thus $\dim \mathcal{M}_0 \le k-1$. 

Next, by \cite[\S2]{semianalytic} and since $\mathcal{M}_0$ is bounded, there exists a finite sequence of disjoint sets $\mathcal{N}_1,\dots, \mathcal{N}_\ell$, each of them is a connected analytic submanifold of dimension at most $k-1$, such that $\mathcal{M}_0=\bigcup_{i=1}^{\ell} \mathcal{N}_i$. 
%Next, by \cite[\S2]{semianalytic} $\mathcal{M}_0$ has a locally finite presentation as a disjoint union of sets $\mathcal{N}_1,\mathcal{N}_2,\dots$, each of them is a connected analytic submanifold of dimension at most $k-1$, and such that\[ i\neq j, \mathcal{N}_i\cap \overline{\mathcal{N}_j}\neq \emptyset\Rightarrow \dim\mathcal{N}_j >\dim \mathcal{N}_i.\]
\end{proof}

\subsection{Non-degeneracy definitions}
\label{nondeg}

Let $U$ be an open subset of $\R^d$, and $\mathcal{L}$ be an affine subspace of $\R^n$. Following \cite{Kleinbock-exponent}, we say that a differentiable map $f: U \to \mathcal{L}$ is \textit{nondegenerate at $x \in U$} if the span of all the partial derivatives of $f$ at $x$ up to some order coincides withthe linear part of $\mathcal{L}$. If $\mathcal{M}$ is a $d$-dimensional submanifold of $\mathcal{L}$, we will say that $M$ is nondegenerate in $\mathcal{L}$ at $y\in \mathcal{M}$ if any (equivalently, some) diffeomorphism $f$ between an open subset $U$ of $\R^d$ and a neighborhood of $y$ in $M$ is nondegeneratein $\mathcal{L}$ at $f^{-1}(y)$. We will say that $f : U \to \mathcal{L}$ (resp., $\mathcal{M} \subset \mathcal{L}$) is nondegenerate in $\mathcal{L}$ if it is nondegenerate in $\mathcal{L}$ at $\lambda_U$-almost every point of $U$, where $\lambda_U$ is the Haar measure on $U$ (resp., of $\mathcal{M}$, in the sense of the smooth measure class on $\mathcal{M}$).

Let $X$ be a metric space, $\mu$ is a measure on $X$ and let $\f:X\to \mathcal{L}$, where $\mathcal{L}$ is an affine subspace in $\R^n$. We say $(f,\mu)$ is \textbf{nonplanar in $\mathcal{L}$} if  for any ball $B$ with $\mu(B)>0$, $\mathcal{L}$ is the intersection of all affine subspaces that contain $f(B\cap\supp{\mu})$. If an analytic map $f : U \to \mathcal{L}$ is nondegenerate in $\mathcal{L}$ then $(f,\lambda_U)$ is nonplanar in $\mathcal{L}$.

\subsection{Divergent orbits for cones}\label{div_cone}

Let $\bG$ be a $\Q$-algebraic Lie group, $G=\bG(\R)$, $\Gamma=\bG(\Z)$, and $X=G/\Gamma$. 
Let $\pi$ be the natural projection $G\rightarrow X$.

\begin{definition}\label{def: divergence}
We say that an orbit $A\pi(g)$ \emph{diverges} if for every compact set $K\subset G/\Gamma$ there is a compact set $\tilde{A}\subset A$ such that $a\pi(g) \notin K$ for every $a\in A\setminus \tilde{A}$. 
\end{definition}

In some cases there is a simple algebraic description for the divergence. 

\begin{definition}\label{def: obvious}
We say that an orbit $A\pi(g)$ diverges in an \emph{obvious} way if there exist finitely many rational representations $\varrho_1,\dots,\varrho_k$ and vectors $v_1,\dots,v_k$, where  $\varrho_j: G\rightarrow\GL(V_j)$ and $v_j\in V_j(\Q)$, such that for any divergent sequence $\{a_i\}_{i=1}^\infty\subset A$ there exist a subsequence $\{a_i'\}_{i=1}^\infty\subset \{a_i\}_{i=1}^\infty$ and an index $1\le j\le k$, such that $\varrho_j\{a_i'g\}v_j\xrightarrow{i\to \infty} 0$. 
\end{definition}

\subsection{The $\Q$-fundamental representations}\label{sec: fundamental representations}

Recall that $S$ is a maximal $\Q$-split torus in $G$. Let $\Phi_\Q$ be the set of roots for $S$ and $\Delta_\Q=\{\alpha_1,\dots,\alpha_r\}$ be a simple system for $\Phi_\Q$. 
Denote by $\chi_1,\dots, \chi_r$ (where $r=\operatorname{rank}_\Q G$) the $\Q$-fundamental weights of $G$. That is, for any $1\le i\le r$, we have 
\begin{equation}\label{eq: fundamental weights}
    \langle\chi_i,\alpha_j\rangle=c_i\delta_{ij}
\end{equation}
for some minimal positive integer $c_i$, where the inner product is defined using the Killing form and $\delta_{ij}$ is the Kronecker delta. 

Let $\varrho_1,\dots,\varrho_r$ be the \emph{$\Q$-fundamental representations of $G$}. That is, for any $i=1,\dots r$ $\varrho_i:G\rightarrow\operatorname{GL}(V_i)$ is a $\Q$-representation with a highest weight $\chi_i$. In particular, the highest weight vector space of $\varrho_i$ for any $i$ is one-dimensional and defined over $\Q$. For any $i$ fix $v_i$ to be a highest weight vector in $V_i(\Q)$. In particular, for any $s\in S$ we have
\begin{equation}\label{eq: highest weight}
    \varrho_i(s)v_i=e^{\chi_i(s)}v_i. 
\end{equation}

For any $1\le i\le r$, the normalizer $P_i$ of $v_i$ in $G$ is a maximal $\Q$-parabolic subgroup. 

Let $T$ be a maximal $\R$-split torus in $G$ which contains $S$, and $\iota:T^*\rightarrow S^*$ be the restriction homeomorphism. 

For any $1\le i\le r$ let $\Phi_{i}$ be the set of $\R$-weights of $\varrho_i$. For any $1\le i\le r$ and $\lambda\in\Phi_{i}$ denote by $V_{\lambda}$ the $\lambda$-weight space in $V_i$. Then, we can decompose 
\begin{equation*}
    V_i=\bigoplus_{\lambda\in\Phi_{i}}V_{\lambda}.
\end{equation*}
For any $1\le i\le r$ and $\lambda\in\Phi_{i}$ denote by $\varphi_{\lambda}$ the projection of $V_i$ onto $V_{\lambda}$. %Note that for any $v\in V_i(\Q)$ and $\lambda\in \Phi_{i}$ so that $V_\lambda$ is defined over $\Q$ we have\begin{equation}\label{eq: weight space proj}    \varphi_{\lambda}(v)\in V_i(\Q). \end{equation}

\subsection{Compactness criterions}

We use the following formulation of the compactness criterion developed in \cite[\S 3]{TW} and further studied in \cite[\S 3]{KleinbockWeiss}.

Recall the definition of $\varrho_i,v_i$ from \S \ref{sec: fundamental representations}. 

\begin{theorem}\label{thm: compactness criterion}
    A set $K\subset X$ is pre-compact if and only if there exists $\eps>0$ such that for any $1\le i\le r$ and $g\in G$ so that $\pi(g)\in X\setminus K$, we have	\begin{equation*}\label{eq: small vectors representation}	\norm{\varrho_i(g)v_i}>\varepsilon.	\end{equation*}
\end{theorem}

We also use the next claim which follows directly from Definition \ref{def: divergence}. 

\begin{lemma}\label{lem: divergence}
    For any $H\subset G$ and $x\in X$, the orbit $Hx$ diverges if and only if for any divergent sequence $\{h_t\}\subset H$ there exists a subsequence $\{h_{t_i}\}$ so that $\{h_{t_i}x\}$ diverges. 
\end{lemma}

\subsection{The relative Weyl group and Bruhat decomposition}\label{sec: Bruhat decomposition}
We follow standard notation and results about Weyl groups and Bruhat decomposition,  see \cite{borel2012, bjorner2006}. 

Recall that $S,T$ are maximal $\Q$-split and $\R$-split tori in $G$.
Let $\Phi_\R$ be the set of roots for $T$ and $\Delta_\R$ be a simple system for $\Phi_\R$. We may choose $\Delta_\R$ so that $\Delta_\Q\subset\iota(\Delta_\R)\subset\Delta_\Q\cup\{0\}$ (see \cite[\S 21.8]{borel2012}). 
Denote by $\Phi_\R^+$ the positive roots in $\Phi_\R$ (with respect to the order induced from $\Delta_\R$). 

Let $W_\R$, $W_\Q$ be the Weyl groups of $T$, $S$, respectively. In particular, $W_\R$ is the groups generated by the reflections $s_\alpha$, $\alpha\in \Delta_\R$ (similarly, $W_\Q$ is generated by $s_\alpha$, $\alpha\in \Delta_\Q$).   
Elements in both sets have representative in the normalizer set of $T$ in $G$, see \cite[Cor. 21.4]{borel2012}. By abuse of notation, we can identify elements of either Weyl groups with such representatives.
In particular, $W_\Q$ can be viewed as a subset of $W_\R$. 

\begin{comment}
Denote by $W_0$ the stabilizer of $\Phi_\Q^+$ in $W_\R$. In particular, $W_0$ is generated by the reflections $s_\alpha$, $\alpha\in\Delta_0:=\{\alpha\in \Delta_\R:\iota(\alpha)=0\}$.

Let $\ell$ be the standard length function on $W_\R$ (see \cite[\S 1.4]{bjorner2006}).  
For any $I\subset \Delta_\R$ let $W_I$ be the subgroup generated by $s_\alpha$, $\alpha\in I$, and let 
\begin{align*}
    W^I&:=\left\{w\in W_\R:\ell(ws_\alpha)>\ell(w)\:\forall \alpha\in I\right\}\\
    \leftindex^{I}W&:=\left\{w\in W_\R:\ell(s_\alpha w)>\ell(w)\:\forall \alpha\in I\right\}.  
\end{align*}

\begin{lemma}[{\cite[Prop. 2.4.4 and (2.12)]{bjorner2006}}]\label{lem: generating Weyl decomp}
    Let $I\subset \Delta_\R$. Then, every $w\in W_\R$ has unique factorizations 
    \begin{align*}
        w &= w^I\cdot w_I,\quad\text{where }w^I\in W^I,w_I \in W_I\\
        w &= w_I\cdot \leftindex^{I} w,\quad\text{ where }w_I \in W_I, \leftindex^{I}w\in \leftindex^{I}W.     
    \end{align*}
    %$\ell(w)=\ell(w^I)+\ell(w_I)$.
\end{lemma}

Let 
\begin{align*}
    \tilde{W}&:= \leftindex^{\Delta_0}W_\R{ }\cap W_\R{ }^{\Delta_\Q}.  
\end{align*}
Then, it follows from Lemma \ref{lem: generating Weyl decomp} that for any $w\in W_\R$ we may decompose 
\begin{equation}\label{eq: Weyl decomp}
    w=w_0\tilde{w}w_\Q,\quad w_0\in W_0,\:\tilde{w}\in,\:w_\Q\in W_\Q. 
\end{equation}

\tcr{Maybe we only need $w=w_0 \tilde{w}$. }
\end{comment}

We follow standard notation and results about algebraic groups,  see \cite{borel2012}. 

Let $B^+$ be the Borel subgroup of $G$ which corresponds to $\Delta_\R$. Note that it normalizes the $w\chi_i$-weight space of $V_i$ for any $i$. Let $B^-$ be the Borel subgroup of $G$ opposite to $B^+$. For any $\lambda\in\Phi_\R$ let $U_\lambda$ be as in \cite[\S 13.18]{borel2012}
For any $w\in W_\R$ let 
\begin{align}
    \Phi_w&:=(\Phi^+)\cap w^{-1}(-\Phi^+), \label{eq: Phi w defn}\\
    U_w^\pm&:=\bigcup_{\lambda\in\pm\Phi_w}U_\lambda.\label{eq: U w defn}
\end{align}
The Bruhat decomposition \cite[\S 14.12]{borel2012} implies that $G=U_w^+ {w}B^+$. By replacing $B^+$ with $B^-$ and taking an inverse, we may deduce the following `opposite' version of the Bruhat decomposition
\begin{equation}\label{eq: Bruhat decomposition}
    G=\bigcup_{w\in W}B^-{w}U_w^-. 
\end{equation}

\section{$W$-singular vectors}

In this section we discuss some basic properties of $W$-singular vectors. 

\begin{lemma}\label{lem: closure and iff}
Let $W$ be a set of weights. If a vector $x$ is $W$-singular, then it is $\overline{W}$-singular. In particular, $\sing (W)=\sing (\overline{W}).$ 
\end{lemma}
%    \item If $W$ is closed, then a vector $x$ is $W$-singular if and only if it is $w$-singular for any $w\in W$. 
\begin{proof}
    This proposition follows from the continuity of the function $w\mapsto\norm{qx-p}_w$ for any $x,p,q$. 
\end{proof}

%For a vector $v=(v_1,\dots,v_d)\in\R^d$ and $1\le i\le d$ let \[
%\hat{v}_i:=(v_1,\dots,v_{i-1},v_{i+1},\dots,v_d).\] 
%\sd{Probably we don't use this notation later. Will remove it}
%The next lemma follows from Definition \ref{defn: singular} and Proposition \ref{prop: closure and iff}.\sd{ I think the following should be $\sing(w)=\sing(\hat w_i)\times \R$, and it follows from the definition of quasi-norm.}
%\begin{lemma}\label{lem: sing non-proper weight}
%    Let $w$ be a weight such that for some $1\le i\le d$, $w_i=0$. Then, a vector $x\in\R^d$ is $w$-singular if and only if $x_i\in\Q$ and $\hat{x}_i$ is $\hat{w}_i$-singular.
%\end{lemma}

%\begin{proof}
%    It is easy to see that if $x\in\R^d$ is $w$-singular, then $x_i\in\Q$ and $\hat{x}_i$ is $\hat{w}_i$-singular. 
    
%    On the other hand, assume $x_i=\frac{p'}{q'}\in\Q$ and $\hat{x}_i$ is $\hat{w}_i$-singular, and let $\delta>0$. Then, there exists $Q'_0$ so that for any $Q>Q'_0$ there exists an integer solution $(p,q)\in\Z^{d+1}$ to \[
%    \norm{q\hat{x}_i-p}_{\hat{w}_i}\le\frac{\delta/(q')^{1/w_{j_0}+1}}{Q},\quad0<q\le Q.\] Here $w_{j_0}$ is the lowest nonzero in $w$. 
%    Then, for any $j\ne i$ we have\[    |q'qx_j-q'p_j|^{1/w_j}\le\frac{\delta}{q'Q}.\]     Since, $q'qx_i-qp'=0$ and $1<q'q\le q'Q$, by taking $Q_0:=Q'_0q'$, we get that $x$ is $w$-singular. \end{proof}

In a similar way to the above we get the following result.  
\begin{lemma}\label{closure and iff uniform}
    For any set of weights $W$ and any vector $x$ we have that $\hat\sigma_W(x)=\hat\sigma_{\overline W}(x)$. 
    
\end{lemma}

%\begin{definition}
%    Given $x\in\R^d$, we say that the \emph{level of rationality of $x$} is the maximum number of rational coefficients (when not all the coefficients are zeros) in the equation \[\alpha_1 x_1+\alpha_2 x_2+\dots+\alpha_d x_d=0.\]
%\end{definition}

Next we have some simple observations about the uniform exponent of vectors in rational hyperplanes. 

\begin{lemma}\label{lem: rational hyperplanes obs}
    Let $w=(w_1,\dots,w_d)$ and $\eps\geq 1$. Then:
    \begin{enumerate}
        \item If $w_1=\dots=w_i=0$ for some $1\le i\le d-1$, then $w':=(w_{i+1},\dots,w_d)$ is a weight of $\R^{d-i}$ and  $\hat{\mathcal{W}}_{w,\eps}=\R^{i}\times \hat{\mathcal{W}}_{w',\eps}$. 
        \item For any $1\le i\le d-1,$ $\Q^i \times\R^{d-i}\subset \hat{\mathcal{W}}_{w,\left(1-\sum_{j=1}^iw_j\right)^{-1}}$. 
    \end{enumerate}
\end{lemma}

\begin{proof}
    The first part of the claim follows from the definition of $w$-singular. The second part follows from Theorem \ref{thm: Dirichlet} used with the weight $\left(1-\sum_{j=1}^iw_j\right)^{-1}(w_{i+1},\dots,w_d)$ of $\R^{d-i}$. 
\end{proof}

\begin{remark}\label{rem: rational hyperplanes} 
    It follows from Theorem \ref{thm: Dirichlet} that any vector which lies on a rational hyperplane in $\R^d$ is in $\hat{\mathcal{W}}_{w,\eps}$ for any proper weight $w$. %Moreover, it follows from the continuity of the $w$-norms that for any closed set of proper weights $W$, all vectors which lie on a rational hyperplane in $\R^d$ are $(\eps,W)$-singular for some $\eps$ (which depends on $W$). 
\end{remark}

\begin{corollary}\label{coro hyperplane}
    Let $\mathcal{L}$ be an affine hyperplane in $\R^d$ and $A$ be a $d\times 1$ parametrizing matrix of $\mathcal{L}$, i.e. for . Let $W$ be a set of proper weights that contains the standard weight. Then $A\notin \Q^d$ if and only if for $\lambda_{\mathcal{L}}$--almost every $y\in\mathcal{L}$ is not $W$-singular. Here $\lambda_{\mathcal{L}}$ is the Lebesgue measure on $\mathcal{L}.$
\end{corollary}
\begin{proof}
In view of , it is clear that if $\lambda_{\mathcal{L}}$-almost every $y\in\mathcal{L}$ is not $W$-singular, then $A\notin \Q^n.$ On the other direction, if $A\notin \Q^n$, then by \cite[Corollary 1.2]{DX}, for $w=(1/d,\cdots,1/d)$ $\lambda_{\mathcal{L}}$--almost every $y\in\mathcal{L}$ is not $w$-singular. Hence, \eqref{eqn: sing(W) and sing(w)} implies the claim.
\end{proof}

\section{Dani's correspondence for higher dimension acting subsemigroups}\label{sec: dani's correspondence}
\label{Dani}

In this subsection we assume $G=\operatorname{SL}_d(\R)$ and let $\mathcal{L}_d$ be the space of unimodular lattices in $\R^d$. Then, there is a natural action of $G$ on $\operatorname{SL}_d(\R)$ by left multiplication. Moreover, this action is transitive and the stabilizer of $\Z^d$ is $\operatorname{SL}_d(\Z)$, implying $\mathcal{L}_d\cong\sl_d(\R)/\sl_d(\Z)$. Denote by $\pi$ the natural projection of $G$ onto $\mathcal{L}_d$, i.e., $\pi(g)=g\Z^d$ for any $g\in G$. For any set of weights $W$ let\[
\mathcal{A}_{W}^{+}:=\{a_{w,t}:=\exp\operatorname{diag}(w_{1}t,\dots,w_{d}t,-t):w\in W, t\ge 0\},\]
When $W=\{w\}$ we denote $\mathcal{A}_{\{w\}}^+$ by $\mathcal{A}_w^+$ for simplicity. For $x\in\R^d$ let \[
u_x:=\begin{pmatrix}
    I_d & x^T\\
    0 &1
    \end{pmatrix}, \]  
where $I_d$ is the $d\times d$ identity matrix. 

\begin{definition}\label{def: div}
    Let $g\in G/\Gamma$ and $A\subset G$. We say that the orbit $Ag$ diverges if for any compact $K\subset G/\Gamma$ there exists $A'\subset A$ so that $A\setminus A'$ is compact and $A'g$ does not intersect $K$. 
\end{definition}

By continuity of the $T$-action on $X$, we have the following statement.

\begin{lemma}
    Let $A\subset T$ and $x\in X$. Then, $Ax$ diverges if and only if $\bar{A}x$ diverges. 
\end{lemma}

We have the following version of Dani's correspondence in our setting, which is proved in a similar way to \cite{Da}. 

\begin{theorem}\label{thm: Dani's correspondence}
Let $W$ be a set of weights and $x\in\R^d$ such that $w_{min}>0$. Then, $x$ is $W$-singular if and only if $\mathcal{A}_W^+\pi(u_x)$ diverges. 
\end{theorem}

\begin{proof}
%First let us see that we may assume $W$ is a closed set. By Lemma \ref{lem: closure and iff} for any point $x$, $x$ is $W$-singular if and only if it is $\overline{W}$-singular. By Definition \ref{def: divergence}, $\mathcal{A}_W^+\pi(u_x)$ diverges if for every compact set $K\subset G/\Gamma$
%there is a compact set $\tilde{A}\subset A$ such that $ax\notin K$ for every $a \in A \setminus \tilde{A}.$ \sda{ I think we don't use the fact that $W$ is closed, so we can remove this para.}

By taking $e^t:=\delta^{-1} Q$, Definition \ref{defn: singular} implies that $x$ is $W$-singular if and only if for every $\delta>0$ there exists $t_0>0$ such that for every $t> t_0$ and $w\in W$ there exists an integer solution $(p,q)\in\N^d\times \Z$ to the system of inequalities 
\begin{equation}\label{eq: div pf sing defn}
    \Vert qx-p\Vert_w\leq \frac{\delta}{e^{t}},\quad 0<q e^{-t}\leq \delta.
\end{equation}

Let $\delta_1:=\delta^{w_{min}}$. Then, for every $\delta_1>0,$ there exists $t_0>0$ such that for every $t>t_0$ and $w\in W$ we have $\delta(a_{w,t} \pi(u_x))<\delta_1$. By Mahler’s compactness criterion (see \cite[\S V]{cassels2012}), the orbit $\mathcal{A}_W^+\pi(u_x)$ diverges. The other direction of the claim follows by following the same arguments in the opposite direction.  
\end{proof}

\begin{remark}
If one wants to consider non-proper weights in Theorem \ref{thm: Dani's correspondence}, then another condition needs to be added to Definition \ref{defn: singular} as follows.  
$\mathcal{A}_W^+\pi(u_x)$ diverges if and only if for every $\delta>0$ there exists $Q_0>0$ such that for every $Q>Q_0$ and $w\in W$ there exists an integer solution $(p,q)\in\Z^d\times\N$ to the system of inequalities 
\begin{equation}
    \norm{qx-p}_{w}\le\frac{\delta}{Q},\quad 0<q\le Q. 
\end{equation}
and for all $1\le i\le d$ such that $w_i=0$, \[
|qx_i-p_i|<\delta. \]

In particular, when considering a set of weights $W$ with non-proper closer, $\mathcal{A}_W^+\pi(u_x)$ is a stronger condition than $x$ being $W$-singular. Thus, we get that for any set of weights, if the orbit $\mathcal{A}_W^+\pi(u_x)$ diverges, then that $x$ is $W$-singular.  
\end{remark}

Next, we show the dynamical interpretation of the $W$-uniform exponent, following ideas in \cite[Thm 3.3]{DFSU}. In what follows, \begin{equation}\label{defn:delta}
\delta(\Lambda):=\min_{0\neq v\in\Lambda}\norm{v}, \Lambda\in \mathcal{L}_d. \end{equation}

The next result shows that it can be used to estimate the $W$-uniform exponent of $x$. 

\begin{theorem}\label{thm: Dani's correspondence:unirofm}
Let $x\in\R^d$. Then 
%\[\hat\sigma_W(x)= \frac{\hat{\tau}_W(x)+w_{\max}}{(1-\hat{\tau}_W(x)) w_{\max}} .\]
\begin{equation}\label{eqn:Dani Uni}
\frac{\hat{\tau}_W(x)+w_{max}}{(1-\hat{\tau}_W(x))w_{max}}\le \hat\sigma_W(x)\leq \frac{\hat{\tau}_W(x)+w_{min}}{(1-\hat{\tau}_W(x)) w_{min}} .
\end{equation}
\end{theorem}
\begin{remark}\label{use_critical}
    Note that the right hand inequality of \eqref{eqn:Dani Uni} is nontrivial when $w_{min}>0$, otherwise it gives $\hat\sigma_W(x)\leq \infty.$ Also note that when $w_{min}>0,$ $\hat\tau_{W}(x)=0$ iff $\hat\sigma_{W}(x)=1.$
\end{remark}

\begin{remark}   
Note that $\hat\tau_W(x)\le 1$ implies that \[\frac{\hat{\tau}_W(x)+w_{max}}{(1-\hat{\tau}_W(x))w_{max}}\le \frac{\hat{\tau}_W(x)+w_{min}}{(1-\hat{\tau}_W(x)) w_{min}} .\] 
Note also that for $w=(1/d,\dots,1/d)$ we have $w_{\min}=w_{\max}=1/d$, and so in this case we get the same conclusion as in \cite[Thm 3.3]{DFSU}. 
\end{remark}

\begin{proof}[Proof of Theorem \ref{thm: Dani's correspondence:unirofm}]
For simplicity, let $\sigma:=\hat\sigma_W(x)$ and \begin{equation}\label{tau_sigma_relation}\tau:=\frac{(\sigma-1)w_{\min}}{(1+
\sigma w_{\min})}.\end{equation} Note that $\sigma\ge 1$ and $w_{\min}\ge0$ imply $0\le\tau\le1$. 

By the definition of $\hat{\sigma}_W(x)$ and for any $\delta>0$, there exists $Q_0>0$ such that for every $Q> Q_0$ and $w\in W$ there exists $(p,q)\in\N^d\times \Z$ such that 
\begin{equation*}
    \Vert qx-p\Vert_w\leq \frac{1}{Q^{\sigma-\delta}},\quad 0<q \leq Q.
\end{equation*} 
If $\tau<1$ fix $\varepsilon=0$, otherwise let $0<\varepsilon<1$. Let $t$ be such that $e^{-t} Q= e^{-(\tau-\varepsilon) t}$, i.e. $Q=e^{t(1-\tau+\varepsilon)}$. In particular, for any $i$ we have \[
e^{w_it}\vert qx_i-p_i\vert<\frac{e^{w_i t}}{Q^{w_i(\sigma-\delta)}}=e^{-[(\sigma-\delta)(1-\tau+\varepsilon)-1]w_it}.\]
Hence, any $t>t_0:=(1-\tau+\varepsilon)^{-1}\log Q_0$ and $w\in W$ satisfy
$$\delta(a_{w,t} \pi(u_x))\leq \max_{1\le i\le d}\left\{e^{-[(\sigma-\delta)(1-\tau+\varepsilon)-1]w_it}, e^{-(\tau -\eps)t} \right\}.$$
Thus, we have $\hat\tau_W(x)\geq \inf_{w\in W}\min_{1\le i\le d}\{((\sigma-\delta)(1-\tau+\varepsilon)-1)w_{i}, \tau-\eps\}$. Since we can take $\eps,\delta$ as small as we want, we get \[
\hat\tau_W(x)\geq \inf_{w\in W}\min_{1\le i\le d}\{(\sigma(1-\tau)-1)w_{i}, \tau\}=\tau,\]
where the last equality holds by the definition of $\tau$. This shows the right hand side inequality in \eqref{eqn:Dani Uni}.
%Therefore, $$\hat\tau_W(x)\geq \min_{i=1}^n\{((1-\tau)\sigma-1)w_{i}, \tau\}=\tau.$$  Let $\tau= ((1-\tau)\sigma-1)w_{min}\implies \tau= (1-\tau)\sigma w_{min}-w_{min}= (\sigma-1)w_{min} -\tau\sigma w_{min}\implies \tau(1+\alpha w_{min})= (\sigma-1)w_{min}\implies \tau= \frac{(\sigma-1)w_{min}}{(1+\sigma w_{min})}.$ 

Next, let $\tau_1:=\hat\tau_W(x)$. Then, for any $\varepsilon>0$ there exists $t_0>0$ such that for any $t>t_0$ and any $w\in W$, \[
\delta(a_{w,t} \pi(u_x))< e^{-(\tau_1-\varepsilon) t}.\] 
This implies that there exists a non-zero integer vector $(p,q)\in\N^d\times\Z$ such that $\vert q\vert \leq e^{t(1-\tau_1+\varepsilon)}$ and  for all $i$\[
e^{w_it}\vert qx_i-p_i\vert<e^{-(\tau_1-\varepsilon) t}, \]
which is equivalent to \[
\vert qx_i-p_i\vert<e^{-(\tau_1-\varepsilon+w_i) t}.\]
Let $Q:=e^{t(1-\tau_1+\varepsilon)}$.  Then $q\leq Q$ and  
\begin{align*}
    \max_{1\le i\le d}\left\{\vert qx_i-p_i\vert^{1/w_i}\right\}&<\max_{1\le i\le d}\left\{e^{-(\tau_1-\varepsilon+w_i) t/w_i}\right\}\\
    &=\max_{1\le i\le d}\left\{Q^{-\frac{\tau_1-\varepsilon}{w_i(1-\tau_1+\varepsilon)}-\frac{1}{1-\tau_1+\varepsilon}}\right\}\\
    &=Q^{-\frac{\tau_1-\varepsilon}{w_{\max}(1-\tau_1+\varepsilon)}-\frac{1}{1-\tau_1+\varepsilon}}.
\end{align*}
This implies $\hat\sigma_W(x)\geq \frac{1+(\tau_1-\varepsilon)/w_{\max} }{1-\tau_1+\varepsilon}$ for every $\varepsilon>0$. Hence,  $\hat\sigma_W(x)\geq \frac{1+\tau_1/w_{max} }{(1-\tau_1)}$. 
\end{proof}

\begin{remark}\label{rem: Dani's correspondence}
Let $w$ be a weight on $\R^d$. Define 
$\hat\tau'_{w}$ similarly to $\hat\tau_{w}$ (see Definition \ref{tauhatset-dynamics}), by replacing $\delta$ with \[ 
\delta_w(\Lambda):=\min_{0\neq v\in\Lambda}\max\{\Vert(v_1,\dots,v_d)\Vert_{w}, \vert v_{d+1}\vert \},\text{ for any lattice }\Lambda\subset\R^{d+1}.\] 
Then, repeating the arguments in the proof of Theorem \ref{thm: Dani's correspondence:unirofm}, one may deduce \[
\hat{\sigma}_w(x)=\frac{1+\hat{\tau'}_w(x)}{1-\hat{\tau'}_w(x)},\]
or equivalently \[
\hat{\tau'}_w(x)=\frac{\hat{\sigma}_w(x)-1}{\hat{\sigma}_w(x)+1}.\]
%A similar result holds for a set of weights when considering the action of $a_{(1/d,\dots,1/d),t}$ and 
%\sda{remove}\[
%\delta_W(\Lambda):=\inf_{w\in W}\min_{0\neq v\in\Lambda}\norm{v}_{w}. %\]
\end{remark}

%\sd{$w$ is a weight  in $\R^d,$ so may be we want $\delta_W(\Lambda):=\min_{0\neq v\in\Lambda}\max\{\Vert\pi(v)\Vert_{w}, \vert v_{d+1}\vert \},$ where $\pi$ is a projection to first $d$ many co-ordinates.}

\section{Inheritance of $W$--Singular vectors and $W$-uniform exponent $\hat\tau_W$}
Let $W\subset (0,1)^d$ be a set of weights that satisfies $w_{min}>0$, which will be our assumption for the rest of the section. Let $u_x$ for $x\in\R^d$ and $\mathcal{A}_{w}^{+}$ are as in  \S \ref{Dani}.
We denote $g^w_t=\operatorname{diag}(e^{w_1t},\cdots, e^{w_dt}, e^{-t}),$ 
which means $\mathcal{A}_{w}^{+}=\{g^w_t, t\geq 0\}.$ By Theorem \ref{thm: Dani's correspondence}, for any proper weight $w$, $x$ is $w$--singular if and only if  $g^w_t \pi(u_x)\to \infty$ as $t\to\infty.$

%The goal of this section is to prove the following inheritance theorems. 

 Let us recall Theorem $2.2$ from \cite{Kleinbock-exponent}, which is an improvement to one of the main theorems in \cite{KM}. The following theorem is referred as quantitative nondivergence as this quantifies the nondivergence results of Margulis in \cite{mar1}. Readers are referred to \cite{KM} and \cite{Kleinbock-exponent} and \cite{KLW} for the definition of Besicovitch space, Federer measure, and good maps.

Using \cite[Theorem 2.2]{Kleinbock-exponent} and Theorem \ref{thm: Dani's correspondence}, we have the following theorem:
   \begin{theorem}\label{Thm_sing}
  Let $X$ be a Besicovitch space, $B=B(x,r)\subset X$ a ball, $\mu$ be a Federer measure on $X$, and suppose that $f:\tilde{B}\to \R^n$ is a continuous map. Suppose that the following two properties are satisfied.
	\begin{enumerate}
		\item\label{c11} For every $w\in W$ and $k>0$, there exists $C,\alpha>0$ such that all the functions $x\to \cov( g^w_ku_{f(x)}\Gamma)$, $\Gamma\in\mathfrak{P}({\Z,n+1})$, are $(C,\alpha)$ good on $\tilde B$ w.r.t. $\mu$;
		\item\label{c22} There exists $c>0$ and $k_i\to \infty$ such that for some $w=w(c,i)\in W$ and any $\Gamma\in\mathfrak{P}({\Z,n+1})$ one has 
		\begin{equation}\label{2nd_cond_R}
		\sup_{B\cap\supp{\mu}}\cov( g^w_{k_i}u_{f(x)}\Gamma)\geq c^{\rk({\Gamma})}.
		\end{equation}
	\end{enumerate}
Then $$\mu\{x\in B~|~f(x) \text{ is $W$-singular }\}=0.$$
\end{theorem}

\begin{proof}
By \cite[Theorem 2.2]{Kleinbock-exponent} for any $1>\varepsilon >0$ we have the following for every $i.$
	\begin{equation*}\label{borel_cantelli}
 \begin{aligned}
	&\mu\left(\bigg\{x\in B\left | \delta (g^{w(c,i)}_{k_i} \pi(u_{f(x)}))<\varepsilon c\right.\bigg\}\right)\\
	&\ll \varepsilon^\alpha \mu(B), \text{ where the implied constant depends on $X, n, \mu$ and $C$,}\\
	&= E  \varepsilon^\alpha.
	\end{aligned}\end{equation*}
 
Hence for every $i$, $\mu\left(\bigg\{x\in B\left | \delta (g^w_{k_i} \pi(u_{f(x)}))<\varepsilon c ~~\forall w\in W\right.\bigg\}\right)\ll \varepsilon^\alpha.$ Since $k_i\to\infty$, for any $1>\varepsilon>0$, and for all large $N$,
$$\mu\left( \bigcap_{k\geq N}\bigg\{x\in B\left | \delta (g^w_{k} \pi(u_{f(x)}))<\varepsilon c ~~\forall w\in W\right.\bigg\}\right)\ll \varepsilon^\alpha.$$
Thus the conclusion of this theorem follows.
    
\end{proof}
\begin{proposition}\label{p2}
Let us take the same notations as in Theorem \ref{Thm_sing}. If Condition \ref{2nd_cond_R} in Theorem \ref{Thm_sing} does not hold, then $f(\supp{\mu}\cap B)$ is contained in the set of $W$--singular vectors.
\end{proposition}
\begin{proof}
	If the second condition does not hold then for every $c>0$, there exists $k_0>0$ such that for every $k>k_0$, and all $w\in W,$ there exists $\Gamma\in \mathfrak{P}({\Z,n+1})$ such that $$
		\sup_{B\cap\supp{\mu}}\cov (g^w_k u_{f(x)}\Gamma)<c^{\rk(\Gamma)}.
		$$ Hence for $\mu$-almost every $x\in B$, for every $c>0$, there exists $k_0>0$ such that for all large $k>k_0$, and for all $w\in W,$
  $$\delta(g^w_k \pi(u_{f(x)}))\leq \cov( g^w_k u_{f(x)}\Gamma )^{\frac{1}{\rk(\Gamma)}}\implies \delta(g^w_k \pi(u_{f(x)}))< c.$$ Now using Theorem \ref{thm: Dani's correspondence}, we can conclude that for $\mu$-almost every $x\in B$ we have $f(x)$ to be $W$--singular.
\end{proof}
\subsection{Covolume calculation}
Let us denote the set of rank $j$  submodules of ${\Z}^{n+1}$ as $\mathcal{S}_{n+1, j}$.

Note that  $u_{x}$ leaves $\be_i\in\R^{n+1}$ invariant for $i=0,\cdots, n-1$ and sends $\be_n$ to $\sum_{i=1}^n x_i\be_{i-1}+\be_n$.
Therefore the subspace $\mathcal{V}^0_n=\{(v_0,\cdots, v_n)~|~ v_n=0\}$ is invariant under $u_{x}.$ Also, $g_t^w\be_i=e^{w_{i+1}t}\be_i$ for $i=0,\cdots, n-1$, and $g_t^w \be_n= e^{-t}\be_n$. Suppose $\mathcal{V}=\R^{n+1}$, and we consider $\bigwedge^j \mathcal{V}$ for $j\geq 1.$ 
Let $\mathbf{v}=\sum v_I \be_I\in \bigwedge^j\Z^{n+1}.$ We can write  $\mathbf{v}= \mathbf{v}_0\wedge (q\be_n-(\mathbf{p},0)),$ where $q\in\Z$, $\mathbf{p}\in\Z^{n}$ and $\mathbf{v}_0\in \bigwedge^{j-1}\Z^{n+1}\cap \bigwedge^{j-1}\mathcal{V}_n^0.$ It can be shown from above that 
$g_t^w u_x\mathbf{v}= g^w_t(\mathbf{v}_0\wedge (q x-\mathbf{p},0)+  qg^w_t(\mathbf{v}_0\wedge \be_n).$
For any subspace $\mathcal{H}$, we denote $d_{\mathcal{H}}$ the distance function from the subspace.
Now if $\mathbf{v}\in \bigwedge^j\Z^{n+1}$ represents $\Gamma\in \mathfrak{P}({\Z,n+1})$, then 
$$\cov(g_t^w u_{f(x)}\Gamma)\asymp\max\{ \Vert g^w_t(\mathbf{v}_0\wedge (q f(x)-\mathbf{p},0)\Vert, \vert q \vert \Vert g^w_t(\mathbf{v}_0\wedge \be_n)\Vert\}.$$
%For $q\neq 0,$ $$\Vert g^w_t(\mathbf{v}_0\wedge (q f(x)-\mathbf{p},0)\Vert= e^{\gamma_w t}\left\Vert\left(\mathbf{v}_0\wedge \sum_{i=1}^n (qf_i(x)-p_i)\be_{i-1}\right)\right\Vert,$$ where $\gamma_w>0$ and depends on $\mathbf{v}_0.$ 
Now for $q\neq 0,$ $$\begin{aligned}
    \Vert g_t^w\mathbf{v}_0\wedge g_t^w(q f(x)-\mathbf{p},0)\Vert &=
\Vert g_t^w\mathbf{v}_0\Vert d_{g_t^w\mathcal{H}} (g_t^w(qf(x)-\mathbf{p},0) )\\&= \vert q\vert \Vert g_t^w\mathbf{v}_0\Vert d_{g_t^w(\mathcal{ H}+{(\mathbf{p},0)}/q)} (g_t^w(f(x),0)) .
\end{aligned}$$ Here $\mathcal{H}$ is the subspace of $\mathcal{V}_0$, that corresponds to $\mathbf{v}_0.$ Let $e^{\gamma t}$ is the smallest eigenvalue of $\{g_t^w, w\in W\}$ on $\bigwedge^j \mathcal{V}_0.$ Since $W$ satisfies $w_{min}>0$, we can guarantee $\gamma>0.$ Let us denote $\mathcal{H}(\mathbf{v}):=\mathcal{H}+(\mathbf{p},0)/q.$
Hence 
\begin{equation}\cov(g_t^w u_{f(x)}\Gamma)\geq e^{\gamma t}  \vert q\vert d_{\mathcal{H}(\mathbf{v})}(f(x), 0).\end{equation}

%For simplicity of the notation, we will drop $`0'$ in the last coordinate in the above.

Note that from the above discussion, we can rewrite Condition $(2)$ in Theorem \ref{Thm_sing} as the following:
     There exists $c>0$ and $k_i\to \infty$ such that for some $w=w(c,i)\in W$ and any $\mathbf{v}= \mathbf{v}_0\wedge (q\be_n-(\mathbf{p},0))\in \bigwedge^j\Z^{n+1},$ where $\mathbf{v}_0\in \bigwedge^j\mathcal{V}_0, (\mathbf{p},q)\in\Z^{n+1},$ one has 
		\begin{equation}
		\sup_{x\in B\cap\supp{\mu}}\max\{ \vert q\vert \Vert g_{k_i}^w\mathbf{v}_0\Vert d_{g_{k_i}^w(\mathcal{H}+(\mathbf{p},0)/q)} (g_{k_i}^wf(x)), \vert q \vert \Vert g^w_{k_i}(\mathbf{v}_0\wedge \be_n)\Vert\}\geq c^{j}.
		\end{equation}

%\sda{Note that for $q\neq 0$, $$\begin{aligned} \Vert (g^w_t\mathbf{v}_0\wedge g^w_t(q f(x)-\mathbf{p},0)\Vert &\geq e^{\gamma t} \Vert \mathbf{v}_0\wedge(q f(x)-\mathbf{p},0)\Vert \\
%& =e^{\gamma t}\Vert \mathbf{v}_0\Vert d_{\mathcal{H}} (qf(x)-\mathbf{p} )\\
%& \geq e^{\gamma t}  d_{\mathcal{H}+\mathbf{p}/q} (f(x)) .   
%\end{aligned}$$}

Let $U$ be an open bounded ball in $\R^d$ and $f:U \subset \R^d\to \mathcal{L}\subset \R^n$ be a continuous map. Without loss of generality, we have $f(U)\subset (0,1)^n.$ Let $\dim\mathcal{L}=s$, Suppose $h:\R^s\to\mathcal{L}\subset \R^n$ be an affine isomorphism, and $h(x)=R x$, where $R$ is a $n\times s$ matrix. Let $g:=h^{-1}\circ f: U \to \R^s$. Since $f$ is nondegenerate in $\mathcal{L}$, $g$ is nondegenerate in $\R^s$. Let us denote $S:= R^{-1}(0,1)^{n}$, which is a simplex in $\R^s$.

\begin{lemma} Let $A$ be a matrix in $GL_{n}(\R)$, $\mathcal{H}$ be an affine subspace in $\R^n$, $M>0$ and $B\subset  U$  be a bounded ball. Then there exists a $c_{g,B}>0$, that depends on both $g$ and $B$, such that,
\begin{equation}\label{Condition fR}
\sup_{y\in S } d_{\mathcal{H}} (A Ry)\geq c_{g, B} M\implies\sup_{x\in B} d_{\mathcal{H}} ( A f(x))\geq M.
\end{equation}
and 
\begin{equation}\label{Condition Rf}
\sup_{x\in B} d_{\mathcal{H}} (A f(x))\geq M \implies \sup_{y\in S } d_{\mathcal{H}} (A Ry)\geq M.
\end{equation}

\end{lemma}

\begin{proof}
Note
    $$
    \sup_{x\in B} d_{\mathcal{H}} (A f(x))=\sup_{x\in B} d_{\mathcal{H}} (A R g(x))= \sup_{y\in g(B)} d_{\mathcal{H}} (A R y) .
    $$

 Suppose $\sup_{x\in B} d_{\mathcal{H}} (A f(x))< M$.  Since $g=(g_1,\cdots,g_s)$ is nondegenerate in $\R^s,$ $1, g_1,\cdots, g_s$ are linear independent over $\R.$ This implies $g(B)=\{g(x)~|~x\in B\}$ contains a basis of $\R^s$, say $\{g(x_1),\cdots, g(x_s)\}$, where $x_1,\cdots, x_s\in B$. By definition $g(B)\subset S.$ For any $y\in S,$ $d_{\mathcal{H}}(A Ry)\leq c_{g, B} \sup_{x\in B} d_{\mathcal{H}} (A R g(x))< c_{g,B} M.$ The last inequality uses the fact that $d_{\mathcal{H}}$ is $\vert \phi_{\mathcal{H}}\vert$, where $\phi_{\mathcal{H}}$ is an affine map. Here $c_{g,B}$ depends on both $g$ and $B.$ The second claim follows because $g(B)\subset S$.

 %Let us take $\tilde D:= \{c_1g(x_1)+\cdots+c_s g(x_s)~|~ 0\leq c_i<2, i=1,\cdots, s\}$, and $D$ to be the largest ball that is contained inside $\tilde{D}$. Also, without loss of generality, we may assume that $g(B)\subset D$. This is because $B$ is bounded, so $f(B)$ is inside some compact ball, and we can take $\tilde{D}$ to be large enough. For $y\in \tilde D$, $d_{\mathcal{H}}(Ry)\leq n \sup_{x\in B} d_{\mathcal{H}} (R g(x))< n M.$ The second claim follows because $g(B)\subset D$. Note that $D$ depends on $B$ and $f.$

\end{proof}

We have the following proposition, which follows from Lemma \ref{Condition Rf} and the discussion above that.

\begin{proposition}\label{Independent condition}
    Suppose $\mu=\lambda_d$ is the Lebesgue measure in $\R^d$. Condition $(2)$ in Theorem \ref{Thm_sing} is equivalent to the following:
     There exists $c>0$, $k_i\to \infty$ such that for some $w=w(c,i)\in W$ and any $\mathbf{v}\in \bigwedge^j\Z^{n+1}$ one has 
		\begin{equation}\label{Condition 2 in R}
		\sup_{y\in S}\max\{ \vert q\vert \Vert g_{k_i}^w\mathbf{v}_0\Vert d_{g_{k_i}^w(\mathcal{H}+(\mathbf{p},0)/q)} (g_{k_i}^w Ry), \vert q \vert \Vert g^w_{k_i}(\mathbf{v}_0\wedge \be_n)\Vert\}\geq c^{j}.
		\end{equation}
	where $\mathbf{v}= \mathbf{v}_0\wedge (q\be_n-(\mathbf{p},0)),$ with $\mathbf{v}_0\in \bigwedge^j\mathcal{V}_0, (\mathbf{p},q)\in\Z^{n+1}$,  $\mathcal{H}$ is the subspace associated with $\mathbf{v}_0$. 
\end{proposition}
\subsection{Proof of Theorem \ref{thm inheritance_singular}} 
Suppose there exists $y\in \mathcal{M}$ such that $y$ is not $W$-singular. Then the Condition \ref{2nd_cond_R} in Theorem \ref{Thm_sing} holds by Proposition \ref{p2}. Since $f_{\star}\lambda_d$ is \textit{decaying} by \cite[Theorem 2.1]{KLW}, by \cite[Lemma 4.3]{KW} Condition (1) in Theorem \ref{Thm_sing}. Therefore, by Theorem \ref{Thm_sing}, For $\lambda_{\mathcal{M}}$-almost every $y\in\mathcal{M}$ are not $W$-singular. By Proposition \ref{Independent condition} Condition \ref{2nd_cond_R} only depends on $R$, we conclude the theorem.
\subsection{ A variant of Condition \ref{2nd_cond_R} to prove Theorem \ref{thm inheritance_uniform}} Since the proof of Theorem \ref{thm inheritance_singular} and Theorem \ref{thm inheritance_uniform} are very similar, we state most of the required statements without any proof. %Let $\tau>0$ and $\sigma>1$ are related by Equation \eqref{tau_sigma_relation}, and $\sigma<\frac{1-w_{min}}{w_{max}-w_{min}}.$
\begin{theorem}\label{Thm_unifrom}
  Let $W$ be a set of weights such that $w_{min}>0$ and $\tau\geq 0$. Let $X$ be a Besicovitch space, $B=B(x,r)\subset X$ a ball, $\mu$ be a Federer measure on $X$, and suppose that $f:\tilde{B}\to \R^n$ is a continuous map. Suppose that the following two properties are satisfied.
	\begin{enumerate}
		\item For every $w\in W$, there exists $C,\alpha>0$ such that all the functions $x\to \cov( g^w_ku_{f(x)}\Gamma)$, $\Gamma\in\mathfrak{P}({\Z,n+1})$, are $(C,\alpha)$ good on $\tilde B$ w.r.t. $\mu$;
		\item For every $\gamma>\tau$, there exists a sequence ${k_i}\to\infty$ such that 
  for some $w=w(\gamma,i)\in W$ and any $\Gamma\in \mathfrak{P}(\Z,n+1)$ one has  
  $$
		\sup_{B\cap\supp{\mu}}\cov (g^w_k u_{f(x)}\Gamma)\geq\left(e^{-\gamma k}\right)^{\rk(\Gamma)}.
		$$
	\end{enumerate}
Then for $\mu$ almost every $x\in B$, $\hat\tau_W(f(x))\leq \tau.$
\end{theorem}
\begin{proof} For any $\varepsilon>0,$ and large $k_i$
\begin{equation*}
 \begin{aligned}
	&\mu\left(\bigg\{x\in B\left | ~\delta (g^{w(c,i)}_{k_i} \pi(u_{f(x)}))<e^{-\gamma k_i}\right.\bigg\}\right)\\
	&\ll \varepsilon^{\alpha} \mu(B)\\
	&= E  \varepsilon^{\alpha}.
	\end{aligned}\end{equation*}
 For every $\varepsilon>0$ and for all large $N$,
$$\mu\{ x\in B~|~\hat\tau_W(f(x))>\gamma\}\leq \mu\left( \bigcap_{k\geq N}\bigg\{x\in B\left | \delta (g^w_{k} \pi(u_{f(x)}))< e^{-\gamma k} ~~\forall w\in W\right.\bigg\}\right)\ll \varepsilon^{\alpha}.$$
Thus the conclusion of this theorem follows.
\end{proof}
Then we also have a variant of Proposition \ref{p2}
\begin{proposition}\label{reverse_uniform}
If the second Condition $(2)$ in Theorem \ref{Thm_unifrom} does not hold for $\tau$, then there exists $\gamma>\tau$, $\hat\tau_W(f(x))>\gamma$ for every $x\in \supp{\mu}\cap B.$ 
\end{proposition}
The rest of the proof is an adaptation of the same method as in the proof of Theorem \ref{thm inheritance_singular}. So we will leave it to an enthusiastic reader.

\section{The abstract theorems and their applications}
The following theorem is a slight modification of \cite[Theorem 2.1]{Weiss2004} and \cite[Theorem 5.1]{KW} which was based on abstracting Khintchine’s classical argument. 

Let $Y$ be a locally compact Hausdorff space on which a noncompact locally compact topological group or semigroup $A$ acts. Let $X$ be a locally compact Hausdorff first countable space such that $\pi: X\to Y$ is a covering.

\begin{theorem}\label{thm: Barak geeralization3}
Let $\{X_i\}, \{X_i'\}$ be  sequences of  subsets of $X$ and $\{A_i\}$ be an embedded sequence of subsets of $A.$
%such that $A\pi(x)$ is divergent for every $x\in \bigcup_i X_i.$ 
Assume the following.  
\begin{itemize}
    \item \textbf{Density}: For every $i$\[
    X_i=\overline{X_i\cap\bigcup_{j\ne i}X_j}. \]
    \item \textbf{Transversality I:} For every $j\neq i$, $X_j=\overline{X_j\setminus X_i}$.
     \item \textbf{Transversality II:} For every $j, i$, $X_j=\overline{X_j\setminus X'_i}$.
    \item \textbf{Local Uniformity:} For every $i,j$, $x\in X_j$, and a compact set $K\subset Y$ there exists a compact $C\subset A_i$ and a neighborhood  $\mathcal{U}$ of $x$ such that for every $a\in A_i\setminus C$ and every $z\in\mathcal{U}\cap X_j$, we have $a\pi(z)\notin K$.
\end{itemize}
Then, there are uncountably many $y\in X\setminus\left(\bigcup_{i} X_i\cup \bigcup_{j} X'_j\right)$ such that for any $i$ the orbit $A_i\pi(y)$ diverges.
\end{theorem}

\begin{proof}
Let \[\mathcal{Z}:=\left\{z\in X\setminus \left(\bigcup_{j} X_j\cup \bigcup_{j} X'_j\right)~|~ A_i\pi(z) \text{ diverges for any }i\right\},\] 
and suppose by contradiction that it is countable, i.e., $\mathcal{Z}=\{z_1, z_2, \cdots\}$. 

First, let us fix an increasing sequence of compact sets that exhaust $X$, i.e., $\{S_k\}$ such that $\cup_k S_k=X$, $S_k\subset \int(S_{k+1})$, and $S_0=\emptyset$. We construct a sequence of open sets $U_1,\dots,U_k,\dots$ in $X$ such that for each $k$, $\overline U_k$ is compact, an increasing sequence of compact sets $\tilde{A}_0,\tilde{A}_1\cdots, \tilde{A}_k,\cdots$  of $A$ such that $\tilde{A}_0\subset A_0,\tilde{A}_1\subset A_1\cdots, \tilde{A}_k\subset A_k,\cdots$, and an increasing sequence of indices $i_1,i_2,\cdots$, so that the following properties are met for $k\geq 1$:\

\begin{enumerate}[label=\alph*]
    \item[(a)] ~$\overline U_{k+1}\subset U_{k}$.
    \item[(b)]~ $U_k\cap \left (X'_k\cup\{z_k\}\right)=\emptyset$  and for every $j<i_k$, $U_k\cap X_j=\emptyset. $
    \item[(c)]~ $X_{i_k}\cap U_k\neq\emptyset$ and for every $h\in X_{i_k}\cap U_k$ and every $a\in A_k\setminus \tilde{A}_k$, we have $a\pi(h)\notin S_k.$
    \item[(d)]~For every $h\in U_k,$ and $a\in(A_{k-1}\cap\tilde{A}_{k})\setminus \int \tilde{A}_{k-1}$ we have $a\pi(h)\notin S_{k-1}.$
\end{enumerate}
%\sda{I think it should be $U_k\cap (X_j'\cup \{z_k\})=\emptyset$ }
Now we claim that such a construction leads to a contradiction. Note that $\cap_k U_k$ is nonempty since $\cap_k \overline U_k\subset \cap_k U_k,$ and $\overline U_k$ is compact. Let $z\in \cap U_k$ and $i\in\N$. Then, by $(b)$ we have $z\notin \mathcal{Z}$ and $z\notin \bigcup_i X_i\cup \bigcup_i X'_i$. Since $\{A_k\}$ is an embedded sequence, for any $k> i$ we have $A_i\subset A_{k-1}$, and so by $(d)$, $A_i\pi(z)$ diverges. %\sda{ Why? $z\notin X_{i_k},$ so (c) can not be used. We should have used (d) with some modification. I think we have to get $\{A_{ik}'\}$ an embedding of $A_i$, and modify the proof. - I changed it, take a look.}
%i.e. $\mathbf{c}\in \mathcal{Z},$ which 
A contradiction to the definition of $\mathcal{Z}$. The conclusion of the theorem follows.

Let us choose $i_1=1$, $\tilde{A}_0=\emptyset$, and fix some $x\in X_1$. By the Local Uniformity assumption, there exists a small enough open neighborhood $U_1$ of $x$ such that $z_1\notin U_1 $ and a compact subset $\tilde{A}_1\subset A_1$ so that for all $z\in X_1 \cap U_1$ and all $a \in A_1\setminus\tilde{A}_1$, we have $a\pi(z) \notin S_1$. Also we take $U_1$ small enough such that $U_1\cap X_1'=\emptyset$, which is possible since  $X_1\not\subset \overline{X_1'}=X_1'$ by Transversality II. Then, the conditions are met for $k=1$. 

Assume that we constructed $\tilde{A}_0,\dots,\tilde{A}_k$, $i_1,\dots,i_k$, and $U_1,\dots,U_k$ which satisfy the above. 
By Density, there exists $i_{k+1}\neq i_k$ such that 
\begin{equation}\label{eq:eta defn}
    \mathbf{\eta}\in X_{i_k}\cap X_{i_{k+1}}\cap U_{k}\neq \emptyset.
\end{equation}
%Moreover, we may assume that $\eta\neq z_{k+1}$. 
Note that by the second part of assumption (b) we have $i_{k+1}>i_k$. 
By Local Uniformity, there exists an open neighborhood $U$ of $\mathbf{\eta}$ with $U\subset \overline U_{k}$ and a compact set $\tilde{A}_{k+1}\subset A_{k+1}$ such that all $h\in X_{i_k+1}\cap U$ and all $a\in A_{k+1}\setminus \tilde{A}_{k+1}$ satisfy $a\pi(h)\notin S_{k+1}$. Since $\eta\in X_{i_k}$, $S_k\subset \int(S_{k+1})$ and $\tilde{A}_{k+1}\setminus \int(\tilde{A}_k)$ is compact, by continuity and (c) there exists a neighborhood  $\tilde{U}\subset U$ of $\mathbf{\eta}$ such that for $h\in \tilde{U}$, and $a\in (A_k\cap\tilde{A}_{k+1})\setminus \int(\tilde{A}_{k})$, we have $a\pi(h)\notin S_k$.  
Now let us define  $U_{k+1}$ to be 
\begin{equation}\label{eq: U_k+1 defn}
    U_{k+1}:=\tilde U\setminus\left( \{z_{k+1}\}\cup X_k' \cup\bigcup_{j<i_{k+1}} X_j \right).
\end{equation}

We are left to verify that the above construction satisfies conditions $(a)-(d)$. Condition $(a)$ is satisfied since $U_{k+1}\subset \tilde U\subset U\subset \overline U_{k}.$ By \eqref{eq: U_k+1 defn}, condition $(b)$ is also satisfied. Next, we claim that $X_{i_{k+1}}\cap U_{k+1}\neq \emptyset$. Note that $\eta\in X_{i_{k+1}}\cap \tilde U$.  Then, by the Transversality conditions, $\eta \in X_{i_{k+1}}\subset \overline {X_{i_{k+1}}\setminus \bigcup_{j<i_{k+1}} X_j\cup X_k'}$. %\sda{ In general, $\overline{A}\setminus \overline{B}\subset \overline{A\setminus B}$ or the other way is not true. Take $A=\Q, B=[0,1]$, it fails. The transversality condition should be strengthened.}

Hence, $$
\tilde U\cap X_{i_{k+1}}\setminus \left(X_k'\bigcup_{j<i_{k+1}} X_j \right) \neq \emptyset.
$$ 
Since $z_{k+1}\notin\bigcup_i X_i$, from the above we can also deduce\[
\tilde U\cap X_{i_{k+1}} \setminus \left(X_k'\cup\{z_{k+1}\}\bigcup_{j<i_{k+1}} X_j \right) \neq \emptyset,\]
which implies $U_{k+1}\cap X_{i_{k+1}}\neq\emptyset$. %\sda{($X_k'$ was missing from above.added it.)}
Thus, the first part of condition $(c)$ is met. The second part of condition $(c)$ is met by the construction of $\tilde{A}_{k+1}$ and $U_{k+1}\subset \tilde U\subset U.$ By the choice of $\tilde{U}$ we have 
$h\in U_{k+1}$, and $a\in (A_{k}\cap\tilde{A}_{k+1})\setminus \int (\tilde{A}_{k})$, we have $a\pi(h)\notin S_k.$ Thus, condition $(d)$ is also satisfied.
\end{proof}

Next, we wish to prove a similar statement which also captures the \emph{rate of growth}. 

\begin{definition}
A \emph{rate of growth} of a locally compact space $Y$ is a collection $\{K(t):t\ge 0\}$ of subsets of $Y$ which satisfy the following properties: 
\begin{itemize}
    \item Exhaust: Any compact subset of $Y$ is contained in $K(t)$ for some $t\ge0$.
    \item Embedded: $t_1<t_2$ implies $K(t_1)\subset \operatorname{int}(K(t_2))$.
    \item Continuous: For any $0\le a\le b \le\infty$, the set $\{(t, x) : x\in K(t),a\le t\le b\}$ is closed in $\R\times Y$. 
\end{itemize}
\end{definition}

\begin{definition}
We say that a trajectory ${a(t)y : t\ge 0}$ is \emph{divergent
with rate given by $\{K(t)\}$} if there exists $t_0$ such that for every $t\ge t_0$ we have
$a(t)y \notin K(t)$.
\end{definition}

The next result is a generalization of \cite[Thm 2.4]{Weiss2004}. 
\begin{theorem}\label{thm: Weiss2004_with rate_new}
Let a one parameter semigroup $A=\{a(t)\}$ be given, together with a rate of growth ${K(t)}$. Let $\{X_i\}, \{X_i'\}$ be sequences of subsets of $X$ which satisfy Transversality I as in Theorem \ref{thm: Barak geeralization3} as well as:
\begin{itemize}
    \item \textbf{Density of Transverse:} For every $i$ there exists an infinite set $J_i$ such that for all $j\in J_i$ we have $X_j=\overline{X_j\setminus X'_i}$ and for all $k$ we have $X_k=\overline{X_k\cap\bigcup_{j\in J_i\setminus\{k\}}X_j}$. 
    \item \textbf{Local Uniformity w.r.t. ${K(t)}$:} for every $i$ and every $x\in X_i$ there exists a neighborhood $U$ of $x$ and $t_0$ such that for every $z \in U\cap X_i$ and every $t>t_0$, $a(t)\pi(z)\notin K(t)$. 
\end{itemize}
Then there exist uncountably many $x_0\in X\setminus\left( \bigcup_{i}X_i\cup\bigcup_{j}X_j'\right)$ such that $A\pi(x_0)$ is divergent with rate given by $\{K(t)\}$. 
\end{theorem}

\begin{proof}
We follow a similar strategy to the previous proof. First, we assume by contradiction that the set of points \[
\mathcal{Z}':=\left\{z\in X\setminus\left(\bigcup_i X_i\cup\bigcup_{j}X_j'\right):A\pi(z)\text{ diverges with rate given by }\{K(t)\}\right\}
\] 
is countable. Next, we construct a set $\mathcal{Z}$, a sequence of open sets $U_1,\dots, U_k,\dots \subset X$, and an increasing sequence of indices $i_1,i_2,\dots$, as well as unbounded sequence of positive numbers $T_1<T_2<\cdots$, so that properties (a),(b) as in the proof of Theorem \ref{thm: Barak geeralization3} hold in addition to the following:
\begin{enumerate}
    \item[(c')] $X_{i_k}\cap U_{k}\neq\emptyset$ and for every $h\in X_{i_k}\cap U_k$ and every $t>T_k$ we have $a(t)\pi(z)\notin K(t)$.
    \item[(d')] For every $k\ge2$, $h\in U_k$, and $t\in[T_{k-1},T_k]$ we have $a(t)\pi(z)\notin K(t)$. 
\end{enumerate}
As in the previous proof, $\cap_k U_k$ is nonempty, and any point in this intersection implies a contradiction to the assumption, proving the claim. 

Let us now construct the needed sequences. Let $i_1$ be the smallest index so that $X_{i_1}=\overline{X_{i_1}\setminus X'_1}$, there exists such $i_1\in J_{1}$ by the Transverse are Dense property. Then, there exists $x_1\in X_{i_1}\setminus X_1'$. Note that by the definition of $\mathcal{Z}'$, $x_1\notin\mathcal{Z}'$. By Local Uniformity w.r.t. $\{K(t)\}$, there exists a small enough neighborhood $U_1'$ of $x$ and a big enough natural number $T_1$ so that every $t\ge T_1$ and $z\in U_1\cap X_{i_1}$ satisfy $a(t)\pi(z)\notin K(t)$. By choosing a possible smaller neighborhood of $x_1$, we may assume that $U_1'\cap(X'_1\cup\{z_1\})=\emptyset$. Let $U_1:=\overline{U_1'}$.  
Then, for $k=1$ properties (a),(b),(c'), and (d') are met. 

Now, assume that we constructed $U_1,\dots,U_k$, $i_1,\dots,i_k$, and $T_1,\dots,T_k$ which satisfy the above. As before, by the Density of Transverse property (In particular, we used $X_{i_k}=\overline{X_{i_k}\cap \bigcup_{j\in J_{k}} X_{j}}$) there exists $i_{k+1}>i_k$ such that \eqref{eq:eta defn} is satisfied (note that $i_{k+1}\in J_k$). %and $\eta\notin X_k'$ \sda{why it is required}. 
By Local Uniformity w.r.t. $\{K(t)\}$, there exists an open neighborhood $U$ of $\eta$ and $T_{k+1}>T_k$ such that $\overline{U}\subset U_k$ and all $t\ge T_{k+1}$ and all $z\in U\cap X_{i_{k+1}}$ satisfy $a(t)\pi(z)\notin K(t)$. Since $\eta\in X_{i,k}$, the subsets\[
\left\{(t,a(t)\pi(\eta)):t\in[T_k,T_{k+1}]\right\},\quad\left\{(t,z):z\in K(t),t\in[T_k,T_{k+1}]\right\}\]
of $\R\times Y$ are disjoint, and by the continuity of $\{K(t)\}$, they are closed. Hence, by the continuity of the action and the compactness of $[T_k,T_{k+1}]$ a small enough neighborhood $\tilde{U}\subset U$ of $\eta$ can be chosen so that all points $z\in\tilde{U}$ and $t\in[T_{k-1},T_k]$ satisfy $a(t)\pi(z)\notin K(t)$. We now define the set $U_{k+1}$ by \eqref{eq: U_k+1 defn} (since $i_{k+1}\in J_k)$). %\sd{In order to show nonemptyness of this the current transversility is not enough. See the proof in page 17. Transversality I is still required}.
As in the previous proof, properties (a), (b), (c'), (d') for $k+1$ are satisfied. This completes the proof. 
\end{proof}

\subsection{Existence of weighted singular totally irrational points}

The goal of this subsection it to show the existence of weighted singular totally irrational points with respect to any weight. 

\begin{proof}[{Proof of Theorem \ref{Main: existence manifold}}]
%In view of \cite[Proposition 3.4]{KMW}, it is enough to consider $\mathcal{M}$ to be a bounded real analytic surface that is not contained inside any proper affine subspace in $\R^d.$ 
Since $\mathcal{M}$ is a real analytic submanifold of $\R^d$ of dimension at least $2$, up to a permutation of the coordinates, it contains a submanifold of the form
\begin{equation}\label{eq: parameterization}
    \{(x, f(x)~|~x\in U\},
\end{equation}
where $U$ is a bounded open set in $\R^2$ and $f:U\rightarrow\R^{d-2}$ is a real analytic function. Let $\{X_i\}$ be an enumeration of the intersections of all subspaces of the form $q\times\R^{d-1}$ and $
R\times q\times\R^{d-2}$, where $q\in\Q$. 
The set of intersections of $\mathcal{M}$ with rational hyperplanes is countable. By Lemma \ref{lem: k-dim manifold}, each such intersection is a finite union of real analytic manifolds. Let $\{X_i'\}$ be the set of all closures of such real analytic manifolds which appear in these unions. 

To prove the claim, we show that $\{X_i\}$ and $\{X_i'\}$ satisfy the hypotheses of Theorem \ref{thm: Barak geeralization3}.

\textbf{Local Uniformity:}
Let us recall that $X_i=H_i\cap \mathcal{M}$, where $H_i$ are the affine rational hyperplanes normal to either $e_1=(1,0,\cdots,0)$ or $e_2=(0,1,0,\cdots,0)$ where $\mathcal{M}$ is taken small enough such that each $X_i$ is connected real analytic curve. For fixed $j$, suppose $H_j=\{(x_1,\cdots, x_n)\in\R^d~|~ x_1=\frac{p}{q}\}$, for some $p/q\in\Q$. Now let us consider $\mathcal{H}:=\{x=(x_1,\cdots, x_{d+1})\in \R^{d+1}~|~qx_1+px_{d+1}=0\}.$ Let \begin{align}
        W_i&:=\left\{w=(w_1,\dots,w_d)\in(0,1)^d\::\:\:\sum_{j=1}^dw_j=1,\:\forall 1\le j\le d,\: w_j\ge \frac{1}{i}\right\}, \label{eq: W_i defn}\\
        A_i&:=\mathcal{A}^+_{{W}_i},  \text{ for } i\geq d. \nonumber
    \end{align}
    Let $\Gamma_{\mathcal{H}}= \Z^{d+1}\cap \mathcal{H}.$ Note that for every $x\in X_j$, $\pi(u_x)z$, $z\in \mathcal{H}$, has the first coordinate to be $0$. Hence for any weight $w\in W_i,$ 
$$
cov(g_t^w\pi(u_x)\Gamma_{\mathcal{H}})\ll e^{-w_1t}\implies \delta(g_t^w\pi(u_x)\Z^{d+1})\ll e^{-w_1t/d}\ll e^{-t/id},
$$ here the implied constant depends on $q, p, x$.
Hence by Mahler's compactness criterion for any compact set $K$ in the space of unimodular lattices in $\R^{d+1}$, $x\in X_j$ there exist neighborhood of $\mathcal{U}$ of $x$ and $t_0>0$ such that for all $t\geq t_0$ any $w\in W_i,$
$$g_t^w\pi(u_x)\Z^{d+1}\notin K.$$
% Here the fact W being compact guarantees that there is a compact subset $C$ of $A$ such that outside $C$, the orbit leaves compact set. 
\textbf{Density:}
This is a weaker condition than \cite[Condition (b), Theorem 1.1]{KMW}. Hence one can find in the proof of \cite[Theorem 1.7]{KMW}, the density condition is already checked.
\textbf{Transversality I and II:}
Both of these transversality conditions follow from \cite[Condition (c), Theorem 1.1]{KMW}    
\end{proof}

Recall that a subset of $\R$ is called \emph{perfect} if it is compact and has no isolated points.
Modifying the proof of \cite[Theorem 1.6]{KMW}, together with applying Theorem \ref{thm: Barak geeralization3}, we have
\begin{theorem}
    Let $n\geq 2$ and $S_1,\cdots, S_k$ be perfect sets of $\R$ such that $S_i\cap \Q $ is dense in $S_i$ for $i=1,2$. Suppose $S=\prod_{i=1}^k S_i.$ Then $S$ contains uncountably many vectors which are $w$-singular for any $w\in (0,1)^d$. 
\end{theorem}
%\begin{proof}
%Let us take $X_i$ and $X_i'$ to be as in the proof of . Density and transversality conditions were established in the same proof. The proof of local uniformity condition is exactly as in \S\ref{local uniformity}

%\end{proof}

\subsection{Existence of  totally irrational points with certain uniform exponents} 
Here we prove Theorem \ref{thm: existance of very singular}. 

\begin{proof}[Proof of Theorem \ref{thm: existance of very singular}]
Let $\mathcal{M}$ be a $k$-dimensional real analytic sub-manifold in $\R^d$ which has no open submanifold that is contained in a rational hyperplane of $\R^d$. Then, up to a permutation of its coordinates, $\mathcal{M}$ can be locally parameterized \begin{equation}
    \{(x, f(x)~|~x\in U\},
\end{equation}
where $U$ is a bounded open set in $\R^k$ and $f:U\rightarrow\R^{d-k}$ is a real analytic function. Let $\{X_i\}$ be an enumeration of the intersections of all subspaces of the form $q_1\times\R\times q_{2}\times \R^{d-k}$, where $q_1\in\Q^j,q_2\in\Q^{k-j-1}$ and $0\le j\le k-1$ (we call them type $j$). 
The set of intersections of $\mathcal{M}$ with rational hyperplanes is countable. By Lemma \ref{lem: k-dim manifold}, each such intersection is a finite union of real analytic manifolds. Let $\{X_i'\}$ be the set of all closures of such real analytic manifolds which appear in these unions. 

Let us check that the properties of Theorem \ref{thm: Weiss2004_with rate_new} are satisfied by these sets. 

\textbf{Transversality I:} If $X_i$ and $X_j$ are of same type, then $X_i\cap X_j=\emptyset.$ Otherwise, without loss of generality we can assume $$X_i=\{(x,q_1,\cdots, q_{k-1}, f(x, q_1,\cdots, q_{k-1}))~|~x\in U_1\},$$ and $$X_j=\{(q'_{1},x,q_2', \cdots, q_k', f(q'_{1},x,q_2', \cdots, q_k'), x\in U_2\} .$$ 
It is easy to see that $X_i\cap X_j$ is either empty or a singleton. Thus $X_i=\overline{X_i\setminus X_j}.$

\textbf{Density of Transverse:}
First, note that by definition, $\dim X_j=1$ for any $j$. Hence, $\dim (X_j\cap X_i')\le 1$. If $\dim (X_j\cap X_i')=0$ then by Lemma \ref{lem: k-dim manifold}, it is a finite union of points, and so $X_j=\overline{X_j\setminus X_i'}$, i.e., $j\in J_i$. If $\dim (X_j\cap X_i')=1$, then by Lemma \ref{lem:k-dim manifold}, $X_j\subseteq X_i'$.  
We may deduce that 
\begin{equation}\label{eq: Ji defn}
    j\in J_i\iff X_j\not\subseteq X_i'. 
\end{equation}

Second, without loss of generality let us fix $i$ and assume $X_m$ is of type $1$, i.e.,  \[
X_m=\R\times q_2\times q_3\times\cdots\times q_k\times \R^{d-k}\cap\mathcal{M},\]
for some $q_2,\dots,q_{k}\in \Q$. 
We wish to show that $X_m=\overline{X_m\cap\bigcup_{j\in J_i\setminus\{m\}}X_j}$. Let $q\in\Q$. If there exists $2\le \ell\le k$ such that 
\begin{equation}\label{eq: intersecting set}
    q\times q_2\times\cdots\times q_{\ell-1}\times\R\times q_{\ell+1}\times\cdots\times q_k\times\R^{d-k}\cap\mathcal{M}\not\subseteq X_i',
\end{equation}
Then, there exists some $j\neq m$ such that $X_j$ is equal to the left hand side of \eqref{eq: intersecting set} and by \eqref{eq: Ji defn} we have $j\in J_i$. Thus, $(q,q_2,\dots,q_k,f(1,q,q_2,\dots,q_k))\in X_m\cap\bigcup_{j\in J_i\setminus\{m\}}X_j$. 
since $X_i'$ is a real analytic manifold, if for some $q\in\Q$ \eqref{eq: intersecting set} is not satisfied for any $\ell$, then for some open neighborhood $U$ of $(q_2,\dots,q_k)$ in $\R^{d-1}$ $q\times U\times\R^{d-k}\cap\mathcal{L}\subseteq X_i'$. 

%\sda{ We can only say $q\times \R\times q_3\times \cdots\times q_k\times \R^{d-k}\cap \mathcal{L}\subset X_i'$, $q\times q_2\times \R\times q_4\times q_k\times \R^{d-k}\cap \mathcal{L}\subset X_i'$ and so on.. In order to get an open set $U$, we need to know that we can move $q_2, \cdots,q_k$ simultaneously. But the previous statement says we can move only in one of these directions at each time.}

Since $\dim(X_i')\le k-1$, by the definition of $X_i'$ there exist at most one such $q$. Then \[
\left(\Q\setminus \{q\}\right) \times q_2\times\cdots\times q_k\times\R^{d-k}\cap\mathcal{L}\subseteq X_m\cap\bigcup_{j\in J_i\setminus\{m\}}X_j. \]
This proves the claim. 

\textbf{Local Uniformity:}
By Lemma \ref{lem: rational hyperplanes obs} we have that any $X_i$ of type $i_0$ satisfies 
\begin{equation}\label{eq: local uni}
    X_i\subset \hat{\mathcal{W}}_{w,\left(1-\sum_{j\in\{1,\dots, k\}\setminus\{ i_0\}}^k w_j\right)^{-1}}\subseteq \hat{\mathcal{W}}_{w,\left(\sum_{j=k}^d w_j\right)^{-1}},
\end{equation}
where the last inequality follows from the assumption on $w$. 
Now, Local Uniformity follows from Remark \ref{rem: Dani's correspondence}. 

Now, Theorem \ref{thm: Weiss2004_with rate_new} together with Remark \ref{rem: Dani's correspondence} imply the claim. 
\end{proof}

%\begin{corollary}For any $1\le k\le d$ we have \[ \dim_H(\hat{\mathcal{W}}_{1/k}^*)\ge k-1. \] \end{corollary}

%\begin{proof}
%Let $1\le k\le d-1$. By Theorem \ref{thm: existance of very singular}, any affine sub-space $V\subseteq\R^d$ of dimension $2\le d-k+1\le d$ which is not contained in any rational affine hyperplane of $\R^d$ intersects $\hat{\mathcal{W}}_{1/k}^*$. Fix such a subspace $V$, and take translates of it by all irrational vectors. since all of these affine hyperplanes intersect $\hat{\mathcal{W}}_{1/k}^*$, the projection $\operatorname{proj}_{V^\perp}(\hat{\mathcal{W}}_{1/k}^*)$ has a full Hausdorff dimension, where $V^\perp$ is the orthogonal compliment of $V$ with respect to the dot product on $\R^d$. Since $V^\perp$ is of dimension $k-1$, we may deduce \[\dim_H(\hat{\mathcal{W}}_{1/k}^*)\ge k-1.\] 
%Let us assume by contradiction that \[\dim_H(\hat{\mathcal{W}}_{1/k}^*)>\frac{k^2}{k+1},\]and let \[V=\left\{(x,0):x\in\R^k,0\in\R^{d-k}\right\}.\]According to -- for a.e. $K\in O_{d,k}$ (by -- metric) we have \[\dim_H(\operatorname{proj}_{KV}(\hat{\mathcal{W}}_{1/k}^*))>\frac{k^2}{k+1}. \]
%\end{proof}

\begin{remark}\label{rem: no assumption on coord}
Note that the assumption about the order of the coordinates of $w$ in Theorem \ref{thm: existance of very singular} is only used in \eqref{eq: local uni}. In order to remove this assumption, one need to make sure that the parametrization in \eqref{eq: parameterization} can be done without permutations of coordinates. This is equivalent to the fact that tangent at almost every point in $\mathcal{M}$ not being parallel to the exes of $x_1,\dots,x_k$.    
\end{remark}

The next result shows that the assumption about $w$ in Theorem \ref{thm: existance of very singular} cannot be omitted, only replaced by the assumption about $\mathcal{M}$ not being parallel to certain exes, as in Remark \ref{rem: no assumption on coord}. 

\begin{proposition}\label{prop: no intersection}
Let $w$ be a proper weight, $\eps>0$, and $1\le k\le d-1$. Then, there exist uncountably many analytic sub-manifold of $\R^d$ of dimension $k$ which are not contained in any rational hyperplanes and do not intersect $\hat{\mathcal{W}}_{w,\left(\sum_{i=k+1}^dw_i\right)^{-1}+\eps}$. 
\end{proposition}

\begin{proof}
Let $w':=\left(\sum_{i=k+1}^dw_i\right)^{-1}(w_{1},\dots,w_{k})$. 
Then, $w'$ is a proper weight, and so the uniform $w'$-exponent of almost every point in $\R^{d-k}$ (with respect to the Lebesgue measure) is $1$ (see \cite{Ca} for the standard norm, and \cite{Guan_Shi} for the general setting). Let $x$ be such a point. Then, for any $\eps>0$ the affine subspace $x\times \R^k$ does not intersect \[\hat{\mathcal{W}}_{w',1+\eps}\times\R^{k}\supseteq \hat{\mathcal{W}}_{w,(1+\eps)\left(\sum_{i=k+1}^{d}w_i\right)^{-1}}.\]Since $\eps$ is arbitrary, we may conclude the claim. 
\end{proof}

\section{Bruhat cells of divergent $S^+$-orbits}

Let $P_0$ be the minimal $\Q$-parabolic group of $G$ which contains $B^+$. The goal of this section is to prove the results in \S\ref{sec:diverging orbits}. Recall the notation of \S\ref{sec:diverging orbits} and \S \ref{sec: Bruhat decomposition}. 

Our first step is to prove the following claim. 

\begin{proposition}\label{prop: needed Bruhat cell}
    Let $g\in B^+ w_1 B^+$, $w_1\in W_{\R}$, and assume that  $S^+\pi(g)$ is a divergent orbit. Then, there exists $w_2\in W_\Q$ such that $w=w_1w_2$ satisfies 
    \begin{equation}\label{eq: good weyl element}
        \{w(\chi_1),\dots,w(\chi_r)\}=\{-\chi_1,\dots,-\chi_r\}. 
    \end{equation}       
\end{proposition}

\begin{lemma}\label{lem: character of div}
    Let $x\in X$ and $s\in S$. If the orbit $\{s^nx\}$ diverges, then there exists a representative $g\in G$, $\pi(g)=x$, such that for some $w\in W_\R$ and all $1\le i\le r$\[
     g\in B^+w B^+\quad\text{ and }\quad w(\chi_i)(a)<0.
    \]
\end{lemma}

\begin{proof}
    Since for any $1\le i\le r$ the set $\varrho_i(\Gamma)v_i$ is discrete, there exists $\eps>0$ so that
    \begin{equation}
        \min\{\norm{\varrho_i(g)v_i}:1\le i\le r,\:g\in G\text{ s.t. }\pi(g)=x\}<\eps. 
    \end{equation}
    Assume $\pi(s^ng)$ diverges. It follows from Theorem \ref{thm: compactness criterion} that there exist $1\le i\le r$, a large enough $n$, and $g\in G$ so that $\pi(g)=x$ and $\norm{\varrho_i(s^ng\gamma)v_i}<\eps$. Since $g\gamma\in B^+ w B^+$ we have
    \begin{align*}
        \varrho_i(g\gamma)v_i&= \varrho_i(p_1wp_2)v_i\in\bigoplus_{\lambda>w(\chi_i)}V_{i,\lambda}.
    \end{align*} 
    It follows that \[
    w(\chi_i)(s)<0.
    \]
\end{proof}

\begin{corollary}\label{cor: good weyl element}
    Let $x\in X$ have a divergent $S^+$-orbit. Then, there exists $w\in W_\R$ such that $g\in B^+wB^+$ satisfies $\pi(g)=x$ and $w$ satisfies \eqref{eq: good weyl element}. 
\end{corollary}

\begin{proof}
By Lemma \ref{lem: character of div} and the definition of $S^+$, for any $i,j$ we have \[
\langle \chi_i,w(\alpha_j) \rangle=\langle w(\chi_i),\alpha_j \rangle<0. \]
Therefore, by writing $w(\alpha_j)=\sum_{k}a_k\alpha_k$, we get \[
a_i=\langle \chi_i,\sum_{i}a_k\alpha_k \rangle<0. \]
That is, $w$ maps the positive $\Q$-roots to the negative $\Q$-roots. But this implies that  $w:\Delta_\Q\mapsto-\Delta_\Q$, which implies \eqref{eq: good weyl element}. 
\end{proof}

\begin{lemma}\label{lem: Bruhat cell of repr}
    Let $g\in G$ and $\gamma\in\Gamma$. If $g\in P_0 w_1P_0$ for some $w_1\in W_\R$, then there exists $w_2\in W_\Q$ such that $g\gamma\in P_0 w_1w_2 P_0$. 
\end{lemma}

\begin{proof}
    By \cite[Thm 21.15]{borel2012} $\bG(\Q)\subset P_0 W_\Q P_0$. Thus, for some $w'_2\in W_\Q$ we have $\gamma\in P_0 w'_2 P_0$.
    Let $N_0$ be the maximal unipotent radicals of $P_0$. Then, according to \cite[Thm 21.15]{borel2012}, $(G, P_0, N, \{s_\alpha:\alpha\in\Delta_\R\})$ is a Tits systems. Hence, using induction on the length of $w_2$ and basic properties of a Tits system (see \cite[\S 14.15]{borel2012}), we may find $w_2\in W_\Q$ which satisfies the claim. 
\end{proof}

\begin{proof}[Proof of Proposition \ref{prop: needed Bruhat cell}]
    By Corollary \ref{cor: good weyl element} $\gamma\in\Gamma$ and $w\in W_\R$ so that $g\gamma\in B^+wB^+$ and $w$ satisfies \eqref{eq: good weyl element}. By the uniqueness if the Bruhat decomposition and Lemma \ref{lem: Bruhat cell of repr}, some $w_2\in W_\Q$ satisfies $w=w_1w_2$, as wanted. 
\end{proof}

\begin{proposition}
    Let $g\in G$ and $s\in S^+$ such that $\varrho_i(s^tg)v_i\rightarrow0$ then, $\varrho_i(g)v_i\in$. 
\end{proposition}

\begin{proof}
    Using - we may write\[
    g=b_1wb_2,\quad b_1,b_2\in B^+, w\in W_\R. \]
    Thus, \[
    \varrho_i(g)v_i\in\bigoplus_{\lambda\ge w(\chi_i)}V_\lambda. \]
    Which implies $w(\chi_i)(s)<0$. Thus $\langle w(\chi_i),\alpha_j\rangle<0$. 
    
\end{proof}

\subsection{Proof of Theorem \ref{thm: only obvious}}

Let $S^+\subset A\subset T$ and $x\in X$ so that $Ax$ is a divergent orbit. 

We show the claim in two steps, first we show that $S^+x$ divergent in an obvious way, and than we use a covering argument to show that $Ax$ divergent in an obvious way. 

%Recall that $\langle\cdot,\cdot\rangle$ is an inner product on $T^*$. 
%Let $g\in G$ such that $x=\pi(g)$. Using the `opposite' Bruhat decomposition with respect to $B^-$ (see \S\ref{sec: Bruhat decomposition}) we may write 
Let $g\in G$ be a representative of $x$, i.e. $\pi(g)=x$. Using the `opposite' Bruhat decomposition with respect to $B^-$ (see \S\ref{sec: Bruhat decomposition}) we may write     
\begin{equation}\label{eq: g bruhat decom}
   g=bwu,\qquad b\in B^-,\: w\in W_\R,\: u\in U^-_w.
\end{equation} 

For any divergent sequence $\{s_t\}\subset S^+$ we have $s_t\pi(g)=s_tbs_{-t}s_t\pi(wu)$. 
Since $b\in B^-$, the sequence $\{s_tb s_{-t}\}$ converges. Hence, it follows from  Lemma \ref{lem: divergence} and the assumption that $S^+\pi(g)$ diverges that 
\begin{equation}\label{eq: wu diverges}
    S^+\pi(wu)\text{ diverges.}
\end{equation}
Note that by \eqref{eq: Phi w defn} and \eqref{eq: U w defn} we have 
\begin{equation}\label{eq: wuw inv in Borel}
    wuw^{-1}\in B^+
\end{equation}

Recall the definition of $\iota$ from \S \ref{sec: Bruhat decomposition}. 

\begin{lemma}
    The orbit $S\pi(wu)$ diverges.  
\end{lemma}

\begin{proof}
    We prove the claim using Lemma \ref{lem: divergence}. That is, let $\{a_t\}\subset S$ be a divergent sequence, then we wish to show that for some subsequence of it, $\{a_{t_i}\}$ the sequence $\{a_{t_i}\pi(wu)\}$ also diverges. We need to consider two cases. 

    First, assume that for some $c>0$ there exists a subsequence $\{a_{t_i}\}\subset\{a_{t}\}$ which satisfies 
    \begin{equation}\label{eq: divergent 1st case}
        \chi_j(a_{t_i})\ge-c \quad\text{for all } 1\le j\le r.
    \end{equation}
    Let $d\in S$ such that $\chi_j(d)\ge c$ for all $1\le j\le r$. Then, $\{da_{t_i}\}$  is a divergent sequence in $S^+$. By Lemma \ref{lem: divergence}, $\{da_{t_i}\pi(wu)\}$ has a diverges subsequence. Hence, so does $\{a_{t_i}\pi(wu)\}$.

    Next, we may assume \eqref{eq: divergent 1st case} is not satisfied. That is, up to replacing $\{a_t\}$ with a subsequence of it, for some $1\le j\le r$ we have
    \begin{equation}\label{eq: chi_j diverge}
        \chi_j(a_t)\rightarrow-\infty. 
    \end{equation}

    It follows from \eqref{eq: wuw inv in Borel} that $wu=wuw^{-1}w\in B^+wB^+$. Therefore by Proposition \ref{prop: needed Bruhat cell} there exists $w_2\in W_\Q$ such that $w'=ww_2$ satisfies \eqref{eq: good weyl element}. In particular, there exists $1\le i\le r$ such that $w'(-\chi_i)=\chi_j$. 
    Fix $v\in\varrho_j(w_2)V_{-\chi_j}(\Q)$ (note that it is also a rational vector since $w_2\in\bG(\Q)$). Then, $\varrho_j(w)v\in V_{\chi_i}$ is a highest weight vector, and so \eqref{eq: wuw inv in Borel} implies that $\varrho_i(a_twuw^{-1}w)v=u\in V_{\chi_i}$. Hence, using \eqref{eq: chi_j diverge} we get 
    \begin{align*}
        \varrho_i(a_twu)v& =\varrho_i(a_twuw^{-1}w)v
        =\varrho_i(a_t)u        =e^{\chi_i(a_t)}u\rightarrow0
    \end{align*}
    as $t\rightarrow\infty$.
\end{proof}

Using \cite[Thm 1.3]{ST} we may conclude that $S\pi(wu)$ diverges in an obvious way. In particular, there exist finitely many rational representations $\varrho_1,\dots,\varrho_k$ and vectors $v_1,\dots,v_k$, where  $\varrho_j: G\rightarrow\GL(V_j)$ and $v_j\in V_j(\Q)$, such that for any divergent sequence $\{a_i\}_{i=1}^\infty\subset S^+$ there exist a subsequence $\{a_i'\}_{i=1}^\infty\subset \{a_i\}_{i=1}^\infty$ and an index $1\le j\le k$, such that $\varrho_j\{a_i'wn\}v_j\xrightarrow{i\to \infty} 0$. Since for any such sequence we have $a_iba_i^{-1}\rightarrow e$, the orbit $S^+\pi(bwn)=S^+x$ also diverges in an obvious way. 

\subsection{Proof of Theorem \ref{thm: exists one-parameter diverges}}

We wish to use Theorem \ref{thm: Barak geeralization3} to prove the claim, and so we need to define sequences of subsets $\{X_i\}$, $\{X_i'\}$, $\{A_i\}$, and show that they satisfy the hypotheses of Theorem \ref{thm: Barak geeralization3}. 

First, let $\{X_i\}$ be an enumeration of the sets $\{P_j g: j= 1,2,\:g\in \textbf{G}(\Q)\}$ (where $P_1,P_2$, and $\textbf{G}(\Q)$ are defined in \S\ref{sec: fundamental representations}) and $\{X_i'\}$ all be the empty set. 
Then, the \textbf{Density} and \textbf{Transversality I} properties follow as in the proof of \cite[Thm 3.9]{Weiss2004}. Since $\{X_i'\}$ are all empty, the \textbf{Transversality II} property is also satisfied.

Next, we need to define $\{A_i\}$ and show the \textbf{Local Uniformity} property. In order to define $\{A_i\}$, we need to state some observation. One is that by replacing $\Delta_{\Q}$ with $-\Delta_{\Q}$ we may assume 
\begin{equation}\label{eq:a- def}
S^+=\left\{ s\in S: \forall i\quad\chi_{i}\ge0\right\}.
\end{equation}
The second is that $S$ is split and of dimension $r$, and so it can be identified with $\R^r$. In particular, we may equip it with a norm. Now, for any $i\in\N$ we let 
\begin{equation}
A_{i}:=\left\{ s\in S^+\::\:\chi_{j}\left(s\right)\leq e^{-1/i}\norm s\text{ for }j=1,2\right\}.
\end{equation}

The sequence $\{A_i\}$ is made of an embedded subsets of $S^+$. Moreover, since any one-parameter subsemigroup of $S^+$ is defined by linear functionals of the $\chi_i$'s, any such subsemigroup is contained in $A_i$ for some $i$. Thus, if we show the local uniformity property, then the claim will follow from Theorem \ref{thm: Barak geeralization3}. 

Fix $i\in\N$, let $\left\{a_{t}\right\}\subset A_i$
be a divergent sequence, and let $g\in P_1\textbf{G}(\Q)\cup P_2\textbf{G}(\Q)$. Then, $\norm{a_{k}}\rightarrow\infty$ as
$k\rightarrow\infty$. Without loss of generality, we may assume $g=pq$, where $p\in P_1$ and $q\in\textbf{G}(\Q)$. Recall the definition of $V(\Q)$ and $\varrho_i$, $v_i$ for $i=1,2$ from \S\ref{sec: fundamental representations}. 
Since $q$ is rational, $\varrho_1(q)V(\Q)=V(\Q)$, and so $v_i=\varrho(q)v\in \varrho_1(q)V(\Q)$. Since $P_1$ is the normalizer of $v_1$ and $p\in P_1$, for some $c>0$ we have $\varrho_1(p)v_1=cv_1$. Then, \[
\norm{\varrho_1(a_tg)v}=c\norm{\varrho_1(a_t)v_i}=ce^{-\chi_i(s)}\norm{v}\rightarrow 0, \]
as $s\rightarrow\infty$. 
Then, by Theorem \ref{thm: compactness criterion} and Lemma \ref{lem: divergence} the orbit $A_i\pi(g)$ diverges.

{\small
\subsection*{Acknowledgements}
Authors thank Ralf Spatzier who encouraged the authors to work on this project. SD was in part supported by the Knut and Alice Wallenberg Foundation and also by the EPSRC
grant EP/Y016769/1. Last but not the least, SD thanks Subhajit for his support when it was needed the most. 
}

 \bibliographystyle{abbrv}
\bibliography{conebib.bib}

\end{document}